\documentclass{amsart}
\usepackage{amsrefs}
\usepackage{amsmath}
\usepackage{amssymb}
\usepackage{color}
\usepackage{mathrsfs}
\usepackage[all]{xy}

\usepackage[hyperindex]{hyperref}
\newcommand\sset[1]{\{#1\}}
\newcommand\set[1]{\{\,#1\,\}}
\newcommand\Iso{\operatorname{Iso}}

\newcommand\go{G^{(0)}}
\newcommand\restr[1]{|_{#1}}
\newcommand\cs{\ensuremath{C^{*}}}
\DeclareMathOperator{\Prim}{Prim}

\newcommand\cc{C_{c}}
\DeclareMathOperator{\supp}{supp}

\newtheorem{thm}{Theorem}[section]

\newtheorem{lemma}[thm]{Lemma}
\newtheorem{prop}[thm]{Proposition}
\newtheorem{cor}[thm]{Corollary}

\theoremstyle{definition}

\theoremstyle{remark}
\newtheorem{remark}[thm]{Remark}
\newtheorem*{notation}{Notation}
\newtheorem{example}[thm]{Example}

\newtheorem{mycomment}{Comment}
{\end{mycomment}\endgroup}

\makeatletter
\def\labelenumi{\textnormal{(\@alph\c@enumi)}}
\def\theenumi{\@alph \c@enumi}
\def\labelenumii{\textnormal{(\@roman\c@enumii)}}
\def\theenumii{\@roman \c@enumii}
\newcount\charno
\def\alphapart#1{\charno=96
\advance\charno by#1\char\charno}

\makeatother
\numberwithin{equation}{section}
\newcount\hours
\newcount\minutes       %   For computing the time of day on
\def\timeofday{%    Must be computed when called if preloaded
\hours=\time
\minutes=\hours
\divide\hours by60
\multiply\hours by60
\advance\minutes by-\hours
\divide\hours by60
\ifnum\hours>9\else0\fi\the\hours:\ifnum\minutes>9\else
0\fi\the\minutes}
\def\predate{\date{\color{red}\bfseries \the\day\ \ifcase\month\or
  January\or February\or March\or April\or May\or June\or July\or
        August\or September\or October\or November\or
           December\fi\ \the\year\ --- \timeofday}}
\let\phi\varphi
\def\intiso(#1){\Iso(#1)^{\circ}}
\newcommand\qs{\kappa}
\newcommand\Ii{\mathscr{I}}
\newcommand\Aa{\mathscr{A}}
\newcommand\Bb{\mathscr{B}}
\newcommand\C{\mathbf{C}}
\newcommand\N{\mathbf{N}}
\newcommand\R{\mathbf{R}}
\newcommand\T{\mathbf{T}}
\newcommand\Z{\mathbf{Z}}

\newcommand\smin{\Sigma^{\min}}
\newcommand\ymax{Y^{\max}}
\newcommand\alphat{\tilde\alpha}
\DeclareMathOperator{\Aut}{Aut}
\DeclareMathOperator{\Ind}{Ind}
\DeclareMathOperator{\lt}{lt}
\newcommand\tZ{\underline{Z}}

\newcommand\II{\mathcal{I}}
\newcommand{\qo}{\mathcal{Q}}
\renewcommand\Re{\operatorname{Re}}
\let\Ii\II
\expandafter\ifx\csname BibSpec\endcsname\relax\else

\BibSpec{collection.article}{%
    +{}  {\PrintAuthors}                {author}
    +{,} { \textit}                     {title}
    +{.} { }                            {part}
    +{:} { \textit}                     {subtitle}
    +{,} { \PrintContributions}         {contribution}
    +{,} { \PrintConference}            {conference}
    +{}  {\PrintBook}                   {book}
    +{,} { }                            {booktitle}
    +{,} { }                            {series}
    +{,} { \voltext}                    {volume}
    +{,} { }                            {publisher}
    +{,} { }                            {organization}
    +{,} { }                            {address}
    +{,} { \PrintDateB}                 {date}
    +{,} { pp.~}                        {pages}
    +{,} { }                            {status}
    +{,} { \PrintDOI}                   {doi}
    +{,} { available at \eprint}        {eprint}
    +{}  { \parenthesize}               {language}
    +{}  { \PrintTranslation}           {translation}
    +{;} { \PrintReprint}               {reprint}
    +{.} { }                            {note}
    +{.} {}                             {transition}
    +{}  {\SentenceSpace \PrintReviews} {review}
}
\BibSpec{article}{%
    +{}  {\PrintAuthors}                {author}
    +{,} { \textit}                     {title}
    +{.} { }                            {part}
    +{:} { \textit}                     {subtitle}
    +{,} { \PrintContributions}         {contribution}
    +{.} { \PrintPartials}              {partial}
    +{,} { }                            {journal}
    +{}  { \textbf}                     {volume}
    +{}  { \PrintDatePV}                {date}
    +{,} { \eprintpages}                {pages}
    +{,} { }                            {status}
    +{,} { \PrintDOI}                   {doi}
    +{,} { available at \eprint}        {eprint}
    +{}  { \parenthesize}               {language}
    +{}  { \PrintTranslation}           {translation}
    +{;} { \PrintReprint}               {reprint}
    +{.} { }                            {note}
    +{.} {}                             {transition}
    +{}  {\SentenceSpace \PrintReviews} {review}
}
\fi

\begin{document}
\title[Primitive ideals of groupoid \cs-algebras]{\boldmath The primitive ideals of some \'etale groupoid \cs-algebras}
\author{Aidan Sims}
\address[A. Sims]{School of Mathematics and Applied Statistics\\
University of Wollongong \\
NSW 2522\\
Australia}

\email{asims@uow.edu.au}

\author{Dana P. Williams}
\address[D. Williams]{Department of Mathematics \\ Dartmouth College
  \\ Hanover, NH
03755-3551}

\email{dana.williams@Dartmouth.edu}



\thanks{We thank Astrid an Huef for helpful conversations. This
  research was supported by the Australian Research Council, the
  Edward Shapiro fund at Dartmouth College, and the Simons
  Foundation.}

\date{\today}

\begin{abstract}
  Consider the Deaconu--Renault groupoid of an action
  of a finitely generated free abelian monoid by local homeomorphisms
  of a locally compact Hausdorff space. We catalogue the primitive ideals
  of the associated groupoid $C^*$-algebra. For a special class of
  actions we describe the Jacobson topology.
\end{abstract}

\maketitle

%\tableofcontents

\section{Introduction}
\label{sec:introduction}

Describing the primitive-ideal space of a $C^*$-algebra is typically quite difficult, but
for crossed products of $C_0(X)$ by abelian groups $G$, a very satisfactory description
is available: for each point $x \in X$ and for each character $\chi$ of $G$ there is an
irreducible representation of the crossed product on $L^2(G\cdot x)$. The map which sends
$(x,\chi)$ to the kernel of this representation is a continuous open map from $X \times
\widehat{G}$ to the primitive-ideal space of $C_0(X) \rtimes G$, and it carries
$(x,\chi)$ and $(y,\rho)$ to the same ideal precisely when $\overline{G\cdot x} =
\overline{G\cdot y}$ and $\chi$ and $\rho$ restrict to the same character of the
stability subgroup $G_x = \set{g : g\cdot x = x}$ \cite{wil:crossed}*{Theorem~8.39}.

Regarding $C_0(X) \rtimes G$ as a groupoid $C^*$-algebra leads to a
natural question: what can be said about the primitive-ideal
  spaces of $C^*$-algebras of Deaconu--Renault groupoids of semigroup
  actions by local homeomorphisms? Examples of groupoids of this sort
  arise from the $\N$-actions by the shift map on the infinite-path
  spaces of row-finite directed graphs $E$ with no sources. The
primitive-ideal spaces of the associated graph $C^*$-algebras were
described by Hong and Szyma\'nski \cite{honszy:jmsj04} building on
Huef and Raeburn's description of the primitive-ideal space of a
Cuntz--Krieger algebra \cite{huerae:etds97}. The description given in
\cite{honszy:jmsj04} is in terms of the graph rather than its
groupoid. Recasting their results in groupoid terms yields a map from
$E^\infty \times \T$ to the primitive-ideal space of $C^*(E)$
along more or less the same lines as described above for group
  actions. But this map is not necessarily open, and the equivalence
relation it induces on $E^\infty \times \T$ is complicated by the fact
that orbits with the same closure need not have the same isotropy in
$\Z^k$.

The complications become greater still when $\N$ is replaced with
  $\N^k$, and the resulting class of $C^*$-algebras is
  substantial. For example, it contains the $C^*$-algebras of graphs
  \cite{kprr:jfa97} and $k$-graphs \cite{kumpas:nyjm00} and
  their topological generalisations \cites{yee:cm06,
    yee:jot07}. However, the results of \cite{ckss:fja14} for
  higher-rank graph algebras suggest that a satisfactory description
of the primitive-ideal spaces of Deaconu--Renault groupoids of $\N^k$
actions might be achievable. Here we take a substantial first step by
producing a complete catalogue of the primitive ideals of the
$C^*$-algebra $C^*(G_T)$ of the Deaconu--Renault groupoid associated
to an action $T$ of $\N^k$ by local homeomorphisms of a locally
compact Hausdorff space $X$. Specifically, there is a surjection
$(x,z) \mapsto I_{x,z}$ from $X \times \T^k$ to
$\Prim(C^*(G_T))$. Moreover, $I_{x,z}$ and $I_{x', z'}$ coincide if
and only if the orbits of $x$ and $x'$ under $T$ have the same closure
and $z$ and $z'$ determine the same character of the interior of the
isotropy of the reduction of $G_T$ to this orbit closure. For a very
special class of actions $T$ we are also able to describe the topology
of the primitive-ideal space of $\cs(G_T)$, but in general we can say
little about it. Indeed, graph-algebra examples show that any
  general description will require subtle adjustments to the
  ``obvious'' quotient topology.

The paper is organised as follows. In Section~\ref{sec:preliminaries}
we establish our conventions for groupoids, and prove that if $G$ is
an \'etale Hausdorff groupoid and the interior $\intiso(G)$ of its
isotropy subgroupoid is closed as well as open, then the natural
quotient $G/\intiso(G)$ is also a Hausdorff \'etale groupoid and there
is a natural homomorphism of $\cs(G)$ onto $\cs(G/\intiso(G))$.

In Section~\ref{sec:deac-rena-group} we consider the Deaconu--Renault
groupoids $G_T$ associated to actions $T$ of $\N^k$ by local
homeomorphisms of locally compact spaces $X$. We state our main
theorem about the primitive ideals of $\cs(G_T)$, and begin its
proof. We first show that $G_T$ is always amenable. We then consider
the situation where $\N^k$ acts irreducibly on $X$. We show that there
is then an open $\N^k$-invariant subset $Y \subset X$ on which the
isotropy in $\N^k \times \N^k$ is maximal. For this set $Y$,
$\intiso(G_T|_Y)$ is closed. We finish
Section~\ref{sec:deac-rena-group} by showing that restriction gives a
bijection between irreducible representations of $\cs(G_T)$ that are
faithful on $C_0(X)$ and irreducible representations of $\cs(G_T|_Y)$
that are faithful on $C_0(Y)$. Our arguments in this section are
special to $\N^k$, and make use of techniques developed
in~\cite{ckss:fja14}.

In Section~\ref{sec:prim-ideal-space} we show that if the subspace $Y$
from the preceding paragraph is all of $X$, then $\cs(G_T)$ is an
induced algebra---associated to the canonical action of $\T^k$ on
$\cs(G_T)$---with fibres $\cs(G_T/\intiso(G_T))$. We use this
description to give a complete characterisation of $\Prim(\cs(G_T))$
as a topological space under the rather strong hypothesis that the
reduction of $G_T/\intiso(G_T)$ to any closed $G_T$-invariant subset
of $Y$ is topologically principal. In
Section~\ref{sec:prim-ideals-csg_t} we complete the proof of our main
theorem. The fundamental idea is that for every irreducible
representation $\rho$ of $\cs(G_T)$ there is a set $Y = Y_\rho$ as
above and an element $z = z_\rho \in \T^k$ for which $\rho$ factors
through an irreducible representation of $\cs(G_T|_Y)$ that is
faithful on $C_0(Y)$ and which in turn factors through evaluation (in
the induced algebra) at $z$.

\subsection*{Standing assumptions}
\label{sec:standing-assumptions}

Throughout this paper, all topological spaces (including topological
groupoids) are second countable, and all groupoids are Hausdorff. By a
homomorphism between $C^*$-algebras, we mean a $*$-homomorphism, and
by an ideal of a $C^*$-algebra we mean a closed, 2-sided ideal. We
take the convention that $\N$ is a monoid under addition, so it
includes~0.

\section{Preliminaries}
\label{sec:preliminaries}

Let $G$ be a locally compact second-countable Hausdorff groupoid with
a Haar system.  For subsets $A, B \subset G$, we write
\begin{equation*}
  AB:=\sset{\alpha\beta\in G:(\alpha,\beta)\in (A\times B)\cap G^{(2)}}.
\end{equation*}
We use the standard groupoid conventions that $G^x = r^{-1}(x)$, $G_x
= s^{-1}(x)$, and $G^x_x = G^x \cap G_x$ for $x \in \go$. If $K\subset
\go$, then the restriction of $G$ to $K$ is the subgroupoid $G\restr
K=\sset{\gamma\in G:r(\gamma),s(\gamma)\in K}$. We will be
particularly interested in the \emph{isotropy subgroupoid}
\begin{equation*}
  \Iso(G)=\sset{\gamma\in G:r(\gamma)=s(\gamma)}=\bigcup_{x\in\go} G^x_x.
\end{equation*}
This $\Iso(G)$ is closed in $G$ and is a group bundle over $\go$.

A groupoid $G$ is \emph{topologically principal} if the units with trivial isotropy are
dense in $\go$.  That is, $\overline{\sset{x\in \go: G^x_x = \sset x}}=\go$. It is worth
pointing out that the condition we are here calling topologically principal has gone
under a variety of names in the literature and that those names have not been used
consistently (see \cite{bcfs:sf14}*{Remark~2.3}).

Recall that $\go$ is a left $G$-space: $\gamma\cdot s(\gamma)=r(\gamma)$.  If $x\in\go$,
then $G\cdot x=r(G_{x})$ is called the \emph{orbit} of $x$ and is denoted by $[x]$.  A
subset $A$ of $\go$ is called \emph{invariant} if $G\cdot A\subset A$.  The quotient
space $G\backslash \go$ (with the quotient topology) is called the orbit space. The
quasi-orbit space $\qo(G)$ of a groupoid $G$ is the quotient of $G\backslash \go$ in
which orbits are identified if they have the same closure.  Alternatively it is the
$T_{0}$-ization of orbit space $G\backslash \go$ (see
\cite{wil:crossed}*{Definition~6.9}). In particular, the quasi-orbit space has the
quotient topology coming from the quotient map $q:\go\to\qo(G)$.

An ideal $I\vartriangleleft C_{0}(\go)$ is called \emph{invariant} if
the corresponding closed set
\begin{equation*}
  C_{I}:=\sset{x\in \go:\text{$f(x)=0$ for all $f\in I$}}
\end{equation*}
is invariant.  If $M$ is a representation of $C_{0}(\go)$ with kernel
$I$, then $C_{I}$ is called the \emph{support of $M$}. We say $C_{I}$
is $G$-irreducible if it is not the union of two proper closed
invariant sets.  For example, orbit closures, $\overline{[x]}$, are
always $G$-irreducible.
\begin{lemma} \label{lem-g-irr} Let $G$ be a second-countable locally
  compact groupoid.  A closed invariant subset $C$ of $\go$ is
  $G$-irreducible if and only if there exists $x\in\go$ such that
  $C=\overline{[x]}$.
\end{lemma}
\begin{proof}
  It suffices to see that every closed $G$-invariant set is an orbit
  closure.  This is a straightforward consequence of the lemma
  preceding \cite{gre:am78}*{Corollary~19} and the observation that
  the orbit space $G\backslash \go$ is the continuous open image of
  $G$ and hence totally Baire.
\end{proof}

\begin{remark} \label{rem-not-minimal} We say that $C_{0}(\go)$ is
  \emph{$G$-simple} if it has no nonzero proper invariant ideals.  So
  $C_{0}(\go)$ is $G$-simple exactly when $\go$ has a dense orbit.
  This is much weaker than the notion of minimality, which requires
  that \emph{every} orbit is dense.
\end{remark}

We also want to refer to a couple of old chestnuts.  Recall that there
is a nondegenerate homomorphism
\begin{equation*}
  V:C_{0}(\go)\to M(\cs(G))
\end{equation*}
given on $f\in \cc(G)$ by
$\big(V(\phi)f\big)(\gamma)=\phi(r(\gamma))f(\gamma)$. In particular,
if $L$ is a nondegenerate representation of $\cs(G)$, then we obtain
an associated representation $M$ of $C_{0}(\go)$ by extension:
$M(\phi)=\bar L(V(\phi))$.  The next result is standard.  A proof in
the case where $G$ is principal can be found in
\cite{cla:jot07}*{Lemma~3.4 and Proposition~3.2}, and the proof goes
through in general \emph{mutatis mutandis}.

\begin{prop} \label{prop-invar-ideals} Let $G$ be a second-countable
  locally compact groupoid with a Haar system.  Let $L$ be a
  nondegenerate representation of $\cs(G)$ with associated
  representation $M$ of $C_{0}(\go)$ as above.  Then $\ker M$ is
  invariant.  If $L$ is irreducible, then the support of $ M$ is
  $G$-irreducible.
\end{prop}

\begin{prop} \label{prop-factors} Let $G$ be a second-countable
  locally compact groupoid with a Haar system.  Let $L$ be a
  nondegenerate representation of $\cs(G)$ with associated
  representation $M$ of $C_{0}(\go)$.  If $F$ is the support of $M$,
  then $L$ factors through $\cs(G\restr F)$.  In particular, if $L$ is
  irreducible, then $L$ factors through $\cs(G\restr{\overline{[x]}})$
  for some $x\in\go$.
\end{prop}
\begin{proof}
  Since $F$ a closed invariant set, $U:=\go\setminus F$ is open and
  invariant.  We have a short exact sequence
  \begin{equation*}
    \xymatrix{0\ar[r]&\cs(G\restr U)\ar[r]^{\iota}&
      \cs(G)\ar[r]^{R}&\cs(G\restr F)\ar[r]&0}
  \end{equation*}
  of \cs-algebras with respect to the natural maps coming from
  extension (by 0) and restriction of functions in $\cc(G)$
  \cite{mrw:tams96}*{Lemma~2.10}.  Since $M$ has support $F$, the
  kernel of $L$ contains the ideal corresponding to $\cs(G\restr U)$,
  so $L$ factors through $C^*(G|_F)$.

  The last assertion follows from Proposition~\ref{prop-invar-ideals}
  and Lemma~\ref{lem-g-irr}.
\end{proof}

When the range and source maps in a groupoid $G$ are open maps (in particular, when $G$
is \'etale), the multiplication map is also open: Fix open $A,B \subseteq G$ and
composable $(\alpha,\beta) \in A \times B$, and suppose that $\gamma_{i}\to \alpha\beta$.
Since $r$ is open, the $r(\gamma_i)$ eventually lie in $r(A)$; say $r(\gamma_i) =
r(\alpha_i)$ with $\alpha_i \in A$. Now $\alpha_i^{-1}\gamma_i \to \beta$, and since $B$
is open, the $\alpha_i^{-1}\gamma_i$ eventually belong to $B$, so that $\gamma_i =
\alpha_i (\alpha_i^{-1}\gamma_i)$ eventually belongs to $AB$; so $AB$ is open.

For the remainder of this note, we specialize to the situation where $G$ is \'etale.
Since $G$ is Hausdorff, this means that $\go$ is clopen in $G$ and that $r : G \to \go$
is a local homeomorphism. Hence counting measures form a continuous Haar system for $G$.
The $I$-norm on $\cc(G)$ is defined by
\[
\|f\|_I = \sup_{x \in \go} \max\Bigl\{\sum_{\gamma \in G_x}
|f(\gamma)|, \sum_{\gamma \in G^x} |f(\gamma)|\Bigr\}.
\]
The groupoid $C^*$-algebra $\cs(G)$ is the completion of $\cc(G)$ in
the norm $\|a\| = \sup\set{\pi(a) : \pi\text{ is an $I$-norm bounded
    $*$-representation}}$. For $x \in \go$ there is a representation
$L^x : \cs(G) \to B(\ell^2(G_x))$ given by $L^x(f)\delta_\gamma =
\sum_{s(\alpha) = r(\gamma)} f(\alpha)\delta_{\alpha\gamma}$. This is
called the (left-)regular representation associated to $x$. The
reduced groupoid $C^*$-algebra $\cs_r(G)$ is the image of $\cs(G)$
under $\bigoplus_{x \in \go} L^x$.

A \emph{bisection} in a groupoid $G$, also known as a $G$-set, is a set $U \subset G$
such that $r,s$ restrict to homeomorphisms on $U$. An important feature of \'etale
groupoids is that they have plenty of open bisections: Proposition~3.5 of
\cite{exe:bbms08} together with local compactness implies that the topology on an \'etale
groupoid has a basis of precompact open bisections.

If $G$ is \'etale, then the homomorphism $V : C_0(\go) \to M\cs(G)$ takes values in
$\cs(G)$ and extends the inclusion $\cc(\go) \hookrightarrow \cc(G)$ given by extension
of functions (by 0). We regard $C_0(\go)$ as a $*$-subalgebra of $\cs(G)$. If $L$ is a
representation of $\cs(G)$, then the associated representation $M$ of $C_{0}(\go)$ is
just the restriction of $L$ to $C_{0}(\go)$.  Thus $\ker M=\ker L \cap C_{0}(\go)$.

We write $\intiso(G)$ for the interior of $\Iso(G)$ in $G$.  Since $G$
is \'etale, $\go \subset \intiso(G)$ and $\intiso(G)$ is an open
\'etale subgroupoid of $G$.

\begin{prop}
  \label{prop-quotient-groupoid}
  Suppose that $G$ is a second-countable locally compact Hausdorff
  \'etale group\-oid such that $\intiso(G)$ is closed in $G$.
  \begin{enumerate}
  \item\label{it:intiso acts} The subgroupoid $\intiso(G)$ acts freely
    and properly on the right of $G$, and the orbit space
    $G/\intiso(G)$ is locally compact and Hausdorff.
  \item\label{it:intiso conjugation} For each $\gamma\in G$, the map
    $\alpha\mapsto \gamma\alpha\gamma^{-1}$ is a bijection from
    $\intiso(G)_{s(\gamma)}$ onto $\intiso(G)_{r(\gamma)}$.
  \item\label{it:intiso normal} For each $x \in \go$, the set
    $\intiso(G)_x$ is a normal subgroup of $G^x_x$.
  \item\label{it:G/intiso} The set $G/\intiso(G)$ is a locally compact Hausdorff
      \'etale groupoid with respect to the operations $[\gamma]^{-1}=[\gamma^{-1}]$
      for $\gamma \in G$, and $[\gamma][\eta]=[\gamma\eta]$ for $(\gamma, \eta) \in
      G^{(2)}$. The corresponding range and source maps are given by $r'([\gamma]) =
      r(\gamma)$ and $s'([\gamma]) = s(\gamma)$.
  \item\label{it:quotient TP} The groupoid $G/\intiso(G)$ is topologically principal.
  \item\label{it:quotient amenable} If $G$ is amenable, then so is
    $G/\intiso(G)$.
  \end{enumerate}
\end{prop}
\begin{proof}
  (\ref{it:intiso acts}) Since $\intiso(G)$ is closed in $G$, it acts
  freely and properly on the right of $G$.  Hence the orbit space is
  locally compact and Hausdorff by
  \cite{muhwil:jams04}*{Corollary~2.3}.

  (\ref{it:intiso conjugation}) Conjugation by $\gamma$ is a
  multiplicative bijection of $\Iso(G)_{s(\gamma)}$ onto
  $\Iso(G)_{r(\gamma)}$.  So it suffices to show that
  \begin{equation}
    \label{eq:1}
    \gamma\intiso(G)\gamma^{-1}\subset \intiso(G)\quad\text{ for all $\gamma \in G$.}
  \end{equation}
  Take $\alpha\in \intiso(G)$ such that $s(\gamma)=r(\alpha)$ and let
  $U$ be an open neighborhood of $\alpha$ in $\intiso(G)$.  Let $V$ be
  an open neighborhood of $\gamma$.  Since $G$ is \'etale, we can
  assume that $U$ and $V$ are bisections with $s(V)=r(U)$.  Since
    the product of open subsets of $G$ is open, $VUV^{-1}$ is an open neighborhood of
  $\gamma\alpha\gamma^{-1}$.  Since $U$ and $V$ are bisections and $U$
  consists of isotropy, $VUV^{-1}$ is contained in $\Iso(G)$.  Hence
  $\gamma\alpha\gamma^{-1}\in\intiso(G)$.

  (\ref{it:intiso normal}) Follows from~(\ref{it:intiso conjugation})
  applied with $\gamma \in \Iso(G)_x$.

  (\ref{it:G/intiso}) The maps $r'$ and $s'$ are clearly well defined.
  Suppose that $(\gamma,\eta)\in G^{(2)}$ and that
  $\gamma'=\gamma\alpha$ and $\eta'=\eta\beta$ with
  $\alpha,\beta\in\intiso(G)$.  Then $\gamma'\eta'
  =\gamma\eta(\eta^{-1}\alpha\eta\beta)$.  But
  $\eta^{-1}\alpha\eta\beta \in \intiso(G)$ by $(b)$.  Hence
  $[\gamma'\eta']=[\gamma\eta]$.  This shows that multiplication is
  well-defined.  A similar argument shows that inversion is
  well-defined.  Since the quotient map is open
  \cite{muhwil:plms395}*{Lemma~2.1}, it is not hard to see that these
  operations are continuous.  For example, suppose that
  $[\gamma_{i}]\to[\gamma]$ and $[\eta_{i}]\to [\eta]$ with
  $(\gamma_{i},\eta_{i})\in G^{(2)}$.  It suffices to see that every
  subnet of $[\gamma_{i}\eta_{i}]$ has a subnet converging to
  $[\gamma\eta]$.  But after passing to a subnet, relabeling, and
  passing to another subnet and relabeling, we can assume that there
  are $\alpha_{i},\beta_{i}\in\intiso(G)$ such that
  $\gamma_{i}\alpha_{i}\to \gamma$ and $\eta_{i}\beta_{i}\to \eta$ in
  $G$ (see \cite{wil:crossed}*{Proposition~1.15}).  But then
  $\gamma_{i}\alpha_{i}\eta_{i}\beta_{i}\to \gamma\eta$, and so
  $[\gamma_{i}\eta_{i}]\to[\gamma\eta]$.

  We still need to see that $G/\intiso(G)$ is \'etale.  Its unit space
  is the image of $\go$ which is open since the quotient map is open.
  So it suffices to show that $r'$ is a local homeomorphism.  Given
  $[\gamma] \in G/\intiso(G)$, choose a compact neighborhood $K$ of
  $\gamma$ in $G$ such that $r\restr K$ is a homeomorphism.  Let $q :
  G \to G/\intiso(G)$ be the quotient map.  Then $q(K)$ is a compact
  neighborhood of $[\gamma]$ and $r'$ is a continuous bijection, and
  hence a homeomorphism, of $q(K)$ onto its image.

  (\ref{it:quotient TP}) Take $b\in G/\intiso(G)$ such that
  $r'(b)=s'(b)$ but $b\not=r'(b)$.  (That is, $b\in \Iso(G/\intiso(G))
  \setminus q(\go)$, but the notation is a bit overwhelming.)  It
  follows that $b=q(\gamma)$ for some $\gamma\in \Iso(G)\setminus
  \intiso(G)$.  Let $U$ be a open neighborhood of $b$. Then
  $q^{-1}(U)$ is an open neighborhood of $\gamma$, so meets $G
  \setminus \Iso(G)$.  Take $\delta\in q^{-1}(U)\setminus \Iso(G)$; so
  $s(\delta)\not= r(\delta)$.  Then $q(\delta)\in U$ and
  $r'(q(\delta))\not = s'(q(\delta))$.  In particular, $q(\delta)$
  does not belong to the interior of the isotropy of the groupoid
  $G/\intiso(G)$.  Thus the interior of the isotropy of $G/\intiso(G)$
  is just $q(\go)$.  Now (\ref{it:quotient TP}) follows from
  \cite{bcfs:sf14}*{Lemma~3.1}.

  (\ref{it:quotient amenable}) To see that $G/\intiso(G)$ is amenable,
  we need to see that $r'$ is an amenable map (see
  \cite{anaren:amenable00}*{Definition~2.2.8}).  If $G$ itself is
  amenable, then $r=r'\circ q$ is amenable.  Thus $r'$ is amenable by
  \cite{anaren:amenable00}*{Proposition~2.2.4}.
\end{proof}

Our analysis of primitive ideals in $C^*$-algebras of Deaconu--Renault groupoids $G$ will
hinge on realising $\cs(G)$ as an induced algebra with fibres $\cs(G/\intiso(G))$. The
first step towards this is to construct a homomorphism $\cs(G) \to \cs(G/\intiso(G))$,
which can be done in much greater generality.

\begin{prop}
  \label{prop-quotient-map}
  Let $G$ be a locally compact Hausdorff \'etale groupoid such that
  $\intiso(G)$ is closed in $G$.  There is a $\cs$-homomorphism
  $\qs:\cs(G)\to \cs(G/\intiso(G))$ such that
  \begin{equation*}
    \qs(f)(b)=\sum_{q(\gamma)=b} f(\gamma)\quad\text{for $f\in \cc(G)$
      and $b\in G/\intiso(G)$.}
  \end{equation*}
\end{prop}
\goodbreak
\begin{proof}
  Lemma~2.9(b) of \cite{mrw:jot87} implies that $\qs$ defines a
  surjection of $\cc(G)$ onto $\cc(G/\intiso(G))$.  It clearly
  preserves involution, and
  \begin{align*}
    \qs(f)*\qs(g)(b)&=\sum_{s'(a)=r'(b)} \qs(f)(a^{-1})\qs(g)(ab)
    = \sum_{s'(a)=r'(b)} \sum_{\substack{q(\gamma)=a\\ q(\delta)=b}} f(\gamma^{-1})g(\gamma^{-1}\delta) \\
    &= \sum_{q(\delta)=b} \sum_{s(\gamma)=r(\delta)} f(\gamma^{-1})
    g(\gamma\delta) = \sum_{q(\delta)=b} f*g(\delta) = \qs(f*g)(b).
  \end{align*}

  It is not hard to see that $\qs$ is continuous in the
  inductive-limit topology (see
  \cite{muhwil:nyjm08}*{Corollary~2.17}). Since the
  $\|\cdot\|_{I}$-norm dominates the \cs-norm, the inductive-limit
  topology is stronger than the \cs-norm topology.  Hence $\qs$
  extends to a $C^*$-homomorphism from $\cs(G)$ to $\cs(G/\intiso(G))$
  as claimed.
\end{proof}

\begin{remark} It is fairly unusual for $\intiso(G)$ to be closed in a
  general \'etale groupoid $G$ (but see
  Proposition~\ref{prop-iso-on-y-closed} and
  \cite{kps:xx14}*{Proposition~2.1}). For example, let $X$ denote the
  union of the real and imaginary axes in $\C$, and let $T : X \to X$
  be the homeomorphism $z \mapsto \overline{z}$. Regarding $T$ as the
  generator of an action of $\N$ by local homeomorphisms, we form the
  associated groupoid
  \[
  G_T = \sset{(t,m,t) : t \in \R, m \in \Z\} \cup \{(z, 2m ,z), (z,
    2m+1, \overline{z}) : z \in i\R, m \in \Z}.
  \]
  Then
  \[
  \intiso(G) = \sset{(z, 2m, z) : z \in X, m \in \Z\} \cup \{(t, 2m+1,
    t) : t \in \R\setminus\{0\}, m \in \Z}
  \]
  is not closed: for example, $(0,1,0) \in
  \overline{\intiso(G)}\setminus\intiso(G)$.

  However, we do not have an example of an \'etale groupoid $G$ which
  acts irreducibly on its unit space and in which $\intiso(G)$ is not
  closed; and \cite{kps:xx14}*{Proposition~2.1} implies that no such
  example exists amongst the Deaconu--Renault groupoids of $\N^k$
  actions that we consider for the remainder of the paper.
\end{remark}

  \section{Deaconu--Renault Groupoids}
  \label{sec:deac-rena-group}

  Given $k$ commuting local homeomorphisms of a locally compact
  Hausdorff space $X$, we obtain an action of $\N^{k}$ on $X$ written
  $n\mapsto T^{n}$ (we do not assume that the $T^n$ are surjective---cf.,
  \cite{exeren:etds07}). The corresponding \emph{Deaconu--Renault
    Groupoid} is the set
  \begin{equation}
    \label{eq:2}
    G_{T}:= \bigcup_{m,n\in\N^{k}}\sset{(x,m-n,y) \in X \times \Z^k \times X :
        T^{m}x = T^{n}y}
  \end{equation}
  with unit space $\go_T = \sset{(x,0,x) : x \in X}$ identified with
  $X$, range and source maps $r(x,n,y) = x$ and $s(x,n,y) = y$, and
  operations $(x,n,y)(y,m,z)=(x,n+m,z)$ and $(x,n,y)^{-1}= (y,-n,x)$.
 For open sets $U,V \subseteq X$ and for $m,n \in \N^k$, we define
  \begin{equation}\label{eq:Zset}
    Z(U,m,n,V) := \sset{(x,m-n,y):\text{$x\in U$, $y\in V$ and
        $T^{m}x=T^{n}y$}}.
  \end{equation}

  \begin{lemma}
    \label{lem-lch-top}
    Let $X$ be a locally compact Hausdorff space and let $T$ be an
    action of $\N^{k}$ on $X$ by local homeomorphisms. The
    sets~\eqref{eq:Zset} are a basis for a locally compact Hausdorff topology
    on $G_T$. The sets $Z(U,m,n,V)$ such that $T^m|_U$ and $T^n|_V$ are
    homeomorphisms and $T^m(U) = T^n(V)$ are a basis for the same topology.
    Under this topology and operations defined above, $G_{T}$ is a locally
    compact Hausdorff \'etale groupoid.
  \end{lemma}
  \begin{proof} When $X$ is compact and the $T^m$ are surjective, this
    result follows immediately from \cite{exeren:etds07}*{Propositions
    3.1~and 3.2}. Their proof is easily modified to show that the
    $Z(U,m,n,V)$ form a basis for a topology on $G_{T}$ when $X$ is
    assumed only to be locally compact and the $T^n$ are not assumed
    to be surjective. It is not hard to see that
    the groupoid operations are continuous in this topology.

    Since the $T^m$ are all local homeomorphisms, each $Z(U,m,n,V)$ is
    a union of sets $Z(U', m, n, V')$ such that $T^m|_{U'}$ and
    $T^n|_{V'}$ are local homeomorphisms. Given $U, V$, we have
    \[
    Z(U, m, n, V) = Z\big(U \cap (T^m)^{-1}(T^mU \cap T^nV), m, n, V
    \cap (T^n)^{-1}(T^mU \cap T^nV)\big).
    \]
    So the sets $Z(U, m, n, V)$ such that $T^m|_U$ and $T^n|_V$ are
    homeomorphisms with $T^mU = T^nV$ form a basis for the same
    topology as claimed.

    To see that this topology is locally compact, let $K_{1}$ and
    $K_{2}$ be compact subsets of $X$.  Then just as in
    \cite{exeren:etds07}*{Proposition~3.2}, the map $(x,y)\mapsto
    (x,p-q,y)$ is continuous from the compact set $\sset{(x,y)\in
      K_{1}\times K_{2}:T^{p}x=T^{q}y}$ onto $Z(K_{1},p,q,K_{2})$.
    Hence the latter is compact in $G_{T}$.  It now follows easily
    that $G_{T}$ is locally compact. It is \'etale because the source map
    restricts to a homeomorphism on any set of the form described in the
    preceding paragraph.
  \end{proof}

  We now state our main theorem, which gives a complete listing of the
  primitive ideals of $C^*(G_T)$; but we need to establish a little
  notation first. Recall that for $x \in X$, the orbit $r((G_T)_x)$ is
  denoted $[x]$. So
  \[
  [x] = \sset{y \in X : \text{$T^m x = T^n y$ for some $ m,n \in
      \N^k$}}.
  \]
  We write
  \[
  H(x) := \bigcup_{\substack{\emptyset \not= U \subset \overline{[x]}
      \\ U\text{ relatively open}}} \sset{m - n : \text{$m,n \in \N^k$
      and $T^m y = T^n y$ for all $y \in U$}}.
  \]
  We write $H(x)^\perp := \sset{z \in \T^k : z^g = 1 \text{ for all }
    g \in H(x)}$. We shall see later that $H(x)$ is a subgroup of
  $\Z^k$, so this usage of $H(x)^\perp$ is consistent with the usual
  notation for the annihilator in $\T^k$ of a subgroup of $\Z^k$.
  Our main theorem is the following.

  \begin{thm}
    \label{thm-mainthm}
    Suppose that $T$ is an action of $\N^{k}$ on a locally compact
    Hausdorff space $X$ by local homeomorphisms. For each $x\in X$ and
    $z\in\T^{k}$, there is an irreducible representation $\pi_{x,z}$
    of $\cs(G_{T})$ on $\ell^{2}([x])$ such that
    \begin{equation}\label{eq:pi_xz formula}
      \pi_{x, z}(f) \delta_{y}
      = \sum_{(u,g,y) \in G_T} z^g f(u,g,y) \delta_u\quad\text{ for all
        $f \in C_c(G_T)$}.
    \end{equation}
    The relation on $X \times \T^k$ given by
    \[
    (x,z) \sim (y,w)\text{ if and only if }\overline{[x]} = \overline{[y]}\text{ and }
    \overline{z}w \in H(x)^\perp
    \]
    is an equivalence relation, and $\ker(\pi_{x, z}) =
    \ker(\pi_{y, w})$ if and only if $(x,z) \sim (y,w)$. The map
    $(x,z) \mapsto \ker\pi_{x,z}$ induces a bijection from $(X \times
    \T^k)/{\sim}$ to $\Prim(C^*(G_T))$.
  \end{thm}

  \begin{remark} A warning is in order. Theorem~\ref{thm-mainthm}
    lists the primitive ideals of $C^*(G_T)$, but it says nothing
    about the Jacobson topology. Example~\ref{eg:notopen} below shows
    that neither the map $(x,z) \mapsto \ker\pi_{x,z}$ nor the induced
    map from $\qo(G_T) \times \T^k$ to
    $\Prim(C^*(G_T))$ is open in general.
  \end{remark}

  \begin{example}\label{eg:notopen}
    Consider the directed graph $E$ with two vertices $v$ and $w$ and
    three edges $e,f,g$ where $e$ is a loop at $v$, $g$ is a loop at
    $w$ and $f$ points from $w$ to $v$.  We use the conventions of
    \cite{honszy:jmsj04}, so the infinite paths in $E$ are $e^\infty$,
    $g^\infty$ and $\{g^n f e^\infty : n = 0, 1, 2, \dots\}$.  There
    are two orbits: $[g^\infty]$ and $[e^\infty]$. The latter is dense
    (because $\lim_{n \to \infty} g^n f e^\infty = g^\infty$),
    while the former is a singleton and is closed. As shown in
    \cite{kprr:jfa97}, $\cs(E)$ is isomorphic to $\cs(G_{T})$
    where $T$ is the shift
    operator on the infinite path space $E^{\infty}$. Hence we can
    apply \cite{honszy:jmsj04} to conclude  that each
    $\ker\pi_{e^\infty, z} \subset \ker\pi_{g^\infty, w}$, and if
    $I_{x,z} := \ker\pi_{x,z}$ for $x \in E^\infty$ and $z \in \T$, we
    have $\overline{\{I_{g^\infty, z}\}} = \{I_{g^\infty, z}\} \cup
    \{I_{e^\infty, w} : w \in \T\}$. So, for example, the set
    $E^\infty \times \{w \in \T : \Re(w) > 0\}$ is open in $E^\infty
    \times \T$, but its image is not open in $\Prim(C^*(E))$; and
    likewise the set $\qo(E) \times \{w : \Re(w) > 0\}$ is open in
    $\qo(G_E) \times \T$, but its image is not open in
    $\Prim(C^*(E))$.
  \end{example}

  The proof of Theorem~\ref{thm-mainthm} occupies this and the next
  two sections, culminating in
  Section~\ref{sec:prim-ideals-csg_t}. Our first order of business is
  to show that $G_T$ is always amenable.

  \begin{lemma}
    \label{lem:amenable}
    Let $G_{T}$ be the locally compact Hausdorff \'etale groupoid
    arising from an action of $T$ of $\N^{k}$ on $X$ by local
    homeomorphisms as above.  Let $c:G_{T}\to \Z^{k}$ be the cocycle
    $c(x,k,y)=k$.  Then both $c^{-1}(0)$ and $G_{T}$ are amenable.
  \end{lemma}
  \begin{proof}
    For each $n\in \N^{k}$, let $F_{n}:=\sset{(x,0,y):T^{n}x=T^{n}y}$.
    Then each $F_{n}$ is a closed subgroupoid containing $\go$, and
    \begin{equation*}
      c^{-1}(0)=\bigcup_{n\in\N^{k}}F_{n}.
    \end{equation*}
    In fact, each $F_{n}$ is also open in $G$: for $(x,0,y)\in F_{n}$
    and any neighborhoods $U$ of $x$ and $V$ of $y$, we have $(x,0,y)
    \in Z(U,n,n,V)\subset F_{n}$.

    Since $N^{k}$ acts by local homeomorphisms, for $x \in X$ the set
    $\sset{y\in X:T^{n}y=T^{n}x}$ is discrete and therefore countable.  It then
    follows from \cite{anaren:amenable00}*{Example~2.1.4(2)} that
    $F_{n}$ is a properly amenable Borel groupoid, and hence Borel
    amenable as in \cite{ren:xx13}*{Definition~2.1}.  Since $F_{n}$ is open in
    $G_{T}$, it has a continuous Haar system (by restriction).  Hence
    it is amenable by \cite{ren:xx13}*{Corollary~2.15}.  It then
    follows from \cite{anaren:amenable00}*{Proposition~5.3.37} that
    $c^{-1}(0)$ is measurewise amenable.  Since $c^{-1}(0)$ is open in
    $G_{T}$, it too is \'etale.  Hence $c^{-1}(0)$ is amenable due to
    \cite{anaren:amenable00}*{Theorem~3.3.7}.

    The amenability of $G_{T}$ now follows from
    \cite{spi:tams12}*{Proposition~9.3}.
  \end{proof}

  Our next task is to understand the interior of the isotropy in
  $G_T$. By definition of the topology on $G_T$ this is the union of
  all the sets $Z(U, m,n, U)$ such that $U \subset X$ is open and
  $T^mx = T^nx$ for all $x \in U$. Our approach is based on that of
  \cite{ckss:fja14}*{Section~4}.

  \begin{lemma}
    \label{lem:SigmaU properties}
    Let $T$ be an action of $\N^{k}$ on $X$ by local homeomorphisms.
    For each nonempty open set $U\subset X$, let
    \begin{equation}
      \label{eq:3}
      \Sigma_{U}:= \sset{(m,n)\in\N^{k}\times\N^{k}:\text{$T^{m}x=T^{n}x$
          for all $x\in U$}}.
    \end{equation}
    Then
    \begin{enumerate}
    \item\label{it:SigmaMonoid} $\Sigma_{U}$ is a submonoid of
      $\N^{k}\times\N^{k}$.
    \item\label{it:SigmaEquiv} $\Sigma_{U}$ is an equivalence relation
      on $\N^{k}$.
    \item\label{it:SigmaNested} If $U\subset V$, then
      $\Sigma_{V}\subset \Sigma_{U}$.
    \item\label{it:SigmaInvariant} For $p\in\N^{k}$ and $U$ open and
      nonempty, we have $\Sigma_{U}\subset \Sigma_{T^{p}U}$.
    \end{enumerate}
  \end{lemma}
  \begin{proof}
    Clearly $(0,0)\in \Sigma_{U}$. Suppose that $(m,n),(p,q)\in
    \Sigma_{U}$.  For $x\in U$ we have
    \begin{equation}\label{eq:4}
      T^{m+p}x=T^{m}T^{p}x=T^{m}T^{q}x=T^{q}T^{m}x=T^{q}T^{n}x=T^{n+q}x.
    \end{equation}
    This proves (\ref{it:SigmaMonoid}). Statements
    (\ref{it:SigmaEquiv})~and~(\ref{it:SigmaNested}) are immediate,
    and~(\ref{it:SigmaInvariant}) follows from the special case
    of~\eqref{eq:4} where $p=q$.
  \end{proof}
  Since our aim is identify the primitive ideals of $\cs(G_{T})$, and
  since Lemma~\ref{lem-g-irr} shows that every irreducible
  representation of $\cs(G_{T})$ factors through the restriction of
  $G_T$ to some $\N^{k}$-irreducible subset, we will often assume that
  $X$ itself (viewed as $\go$) is $\N^{k}$-irreducible. In this case,
  we will say that $T$ \emph{acts irreducibly}. Lemma~\ref{lem-g-irr}
  then implies that $X$ has a dense orbit: $X = \overline{[x]}$ for
  some $x \in X$.

  \begin{lemma}
    \label{lem:common open}
    Let $T$ be an $\N^{k}$-irreducible action on $X$ by local
    homeomorphisms. For all open subsets $U, V \subseteq X$, there
    exists a nonempty open set $W$ such that $\Sigma_{U}\cup \Sigma_{V}
    \subset\Sigma_{W}$
  \end{lemma}
  \begin{proof}
    Fix $x$ with $\overline{[x]} = X$. Choose $y\in U$ and $z\in V$
    such that $T^{r}y=T^{s}z$ and $T^{p}z=T^{l}x$.  Then
    $T^{r+l}y=T^{p+s}z$, so $m = r+l$ and $n = s+p$ satisfy
    $T^{n}U\cap T^{m}V\not=\emptyset$.  Since $T^{m}$ and $T^{n}$ are
    local homeomorphisms, and therefore open maps, $W := T^{m}U\cap
    T^{n}V$ is open. Parts
    (\ref{it:SigmaNested})~and~(\ref{it:SigmaInvariant}) of
    Lemma~\ref{lem:SigmaU properties} show that $\Sigma_{U}\subset
    \Sigma_{T^{m}U} \subset \Sigma_{W}$ and $\Sigma_V \subset
    \Sigma_{T^nV} \subset \Sigma_W$.
  \end{proof}

  Given $X$ and $T$ as in Lemma~\ref{lem:common open}, let
  \begin{equation}
    \label{eq:5}
    \Sigma :=\bigcup_{\text{$\emptyset \not= U\subset X$ open}} \Sigma_{U}.
  \end{equation}
  We give $\N^{k}\times \N^{k}$ the usual partial order as a subset of
  $\N^{2k}$:
  \begin{equation*}
    \big((n_{i})^k_{i=1},(n_{i}')^k_{i=1}\big)\le
    \big((m_{i})^k_{i=1},(m_{i}')^k_{i=1}\big)
    \quad\text{if $n_{i}\le m_{i}$ and $n_{i}'\le m_{i}'$ for all $i$.}
  \end{equation*}
  We let $\Sigma^{\min}$ denote the collection of minimal elements
  of $\Sigma \setminus \{(0,0)\}$ with respect to this order.

  \begin{lemma}
    \label{lem:Sigma properties}
    Let $T$ be an irreducible action of $\N^{k}$ by local
    homeomorphisms on a locally compact space $X$, and let $\Sigma$
    and $\smin$ be as above.  Then $\Sigma$ is a submoniod of
    $\N^{k}\times\N^{k}$ and an equivalence relation on $\N^{k}$.  We
    have $\Sigma = (\Sigma - \Sigma) \cap (\N^k \times \N^k)$.
    Furthermore, $\smin$ is finite and generates $\Sigma$ as a monoid.
  \end{lemma}
  \begin{proof}
    We have $(0,0)\in\Sigma_{X} \subset \Sigma$.  If $(m,n), (p,q) \in
    \Sigma$, then there are nonempty open sets $U$ and $V$ such that
    $(m,n)\in\Sigma_{U}$ and $(p,q)\in\Sigma_{V}$.
    Lemma~\ref{lem:common open} yields an open set $W$ with $(m,n),
    (p,q) \in \Sigma_{W}$. Now $(m+p,n+q)\in\Sigma_{W}\subset \Sigma$
    by Lemma~\ref{lem:SigmaU properties}(\ref{it:SigmaMonoid}), so
    $\Sigma$ is a monoid.

    To see that $\Sigma$ is an equivalence relation, observe that it
    is reflexive and symmetric because each $\Sigma_{U}$ is. Consider
    $(m,n), (n,p) \in \Sigma$; say
    $(m,n)\in\Sigma_{U}$ and $(n,p)\in\Sigma_{V}$.  By
    Lemma~\ref{lem:common open}, there is open set $W$ with $(m,n),
    (n,p) \in \Sigma_W$. Hence $(m,p)\in\Sigma_{W}\subset\Sigma$ by
    Lemma~\ref{lem:SigmaU properties}(\ref{it:SigmaEquiv}).

    The containment $\Sigma \subset (\Sigma - \Sigma) \cap (\N^k
    \times \N^k)$ is trivial because $(0,0) \in \Sigma$ and $\Sigma
    \subset \N^k \times \N^k$.  For the reverse containment, suppose
    that $(m,n), (p,q) \in \Sigma$ and $m-p, n-q \in \N^k$. By
    Lemma~\ref{lem:common open} we may choose an open $W$ such that
    $(m,n), (p,q) \in \Sigma_W$. Fix $x \in T^{p+q}W$, say $x =
    T^{p+q}y$.  Lemma~\ref{lem:SigmaU properties} implies first that
    $(q,p) \in \Sigma_W$, and then that $(m+q, n+p) \in
    \Sigma_W$. Hence
    \[
    T^{m-p}x = T^{m-p}(T^{p+q}y) = T^{m+q}y = T^{n+p}y =
    T^{n-q}(T^{q+p}y) = T^{q-n}x.
    \]
    So $(m-p, n-q) \in \Sigma_{T^{p+q}W} \subset \Sigma$.

    Now we argue as in
    \cite{ckss:fja14}*{Proposition~4.4}.\footnote{Though in
      \cite{ckss:fja14}*{Proposition~4.4}, the crucial use, in the
      induction, of the fact that $\Sigma = (\Sigma - \Sigma) \cap
      (\N^k \times \N^k)$ is not made explicit.} Dickson's Lemma
    \cite{rosgar:finitely99}*{Theorem~5.1} implies that $\smin$ is
    finite. We must show that each $(m,n) \in \Sigma$ is a finite sum
    of elements of $\smin$. We argue by induction on $|(m,n)| :=
    \sum^k_{i=1} m_i + n_i$. If $|(m,n)| = 0$, the assertion is
    trivial. Now take $(m,n) \in \Sigma \setminus \{0\}$, and suppose
    that each $(p,q) \in \Sigma$ such that $|(p,q)| < |(m,n)|$ can be
    written as a finite sum of elements of $\smin$. Since
    $(m,n)\not=0$, by definition of $\smin$ there exists $(a,b) \in
    \smin$ such that $(a,b) \le (m,n)$. The preceding paragraph shows
    that $(p,q) = (m,n) - (a,b) \in (\Sigma - \Sigma) \cap (\N^k
    \times \N^k) = \Sigma$. The induction hypothesis implies that
    $(p,q)$ is a finite sum of elements of $\smin$, and then so is
    $(m,n) = (p,q) + (a,b)$.
  \end{proof}

  We let
  \begin{equation}\label{eq:H(T),Ymax}
    \begin{split}
      H(T)=\set{m-n : (m,n)\in\Sigma}\quad\text{and}\\
      \ymax:=\bigcup \set{Y\subset X:\text{$Y$ is open and
          $\Sigma_{Y}=\Sigma$}}.
    \end{split}
  \end{equation}
  \begin{lemma}
    \label{lem-pos-cone}
    Let $T$ be an irreducible action of $\N^k$ by local homeomorphisms
    of a locally compact Hausdorff space $X$. With $\Sigma$ as
    in~\eqref{eq:5}, we have
    \begin{equation}
      \label{eq:6}
      \Sigma=\sset{(m,n)\in\N^{k}\times\N^{k}: m-n\in H(T)}.
    \end{equation}
    The set $\ymax$ is nonempty and open, and is the maximal open set
    in $X$ such that $\Sigma_{\ymax}=\Sigma$.  We have
    $T^{m}\ymax\subset \ymax$ for all $m\in\N^{k}$.
  \end{lemma}
  \begin{proof}
    By definition, $\Sigma\subset \sset{(m,n):m-n\in H(T)}$.  For the
    reverse inclusion, suppose that $m-n=p-q$ with $(p,q)\in\Sigma$.
    Let $g=m-p\in\Z^{k}$.  Fix $a,b\in\N^{k}$ such that $g=a-b$.  Then
    both $(p+a,q+a)$ and $(b,b)$ belong to $\Sigma$.  Hence
    Lemma~\ref{lem:Sigma properties} implies that
    \begin{equation*}
      (m,n)=(p+g,q+g)=(p+a,q+a)-(b,b)\in (\Sigma - \Sigma) \cap
      (\N^{k}\times\N^{k})=\Sigma.
    \end{equation*}
    This establishes \eqref{eq:6}.

    Now $|\smin| - 1$ applications of Lemma~\ref{lem:common open} give
    a nonempty open set $Y$ such that $\smin\subset \Sigma_{Y}$.  Since
    $\Sigma_{Y}$ is monoid by Lemma~\ref{lem:SigmaU properties}, we
    have $\Sigma_{Y}=\Sigma$ by Lemma~\ref{lem:Sigma properties}.

    It now follows that $\ymax$ is open and nonempty. It is clearly
    maximal. Each $T^{m}\ymax\subset\ymax$ by Lemma~\ref{lem:SigmaU
      properties}(\ref{it:SigmaInvariant}) and the definition of
    $\ymax$.
  \end{proof}

  \begin{prop}
    \label{prop-iso-on-y-closed}
    Let $T$ be an irreducible action of $\N^k$ by local homeomorphisms
    of a locally compact Hausdorff space $X$, and let $G_{T} $ be the
    associated Deaconu--Renault groupoid (as in~\eqref{eq:2}). The set
    $H(T)$ of~\eqref{eq:H(T),Ymax} is a subgroup of $\Z^k$. Let $\Sigma$
    be as in~\eqref{eq:5}, and let $Y \subset X$ be an open set such that
    $\Sigma_{Y}=\Sigma$ and $T^{p}Y\subset Y$ for all $p\in\N^{k}$.
    Then $\intiso(G_{T}\restr Y)=\sset{(x,g,x):\text{$x\in Y$ and
        $g\in H(T)$}}$, and $\intiso(G_{T}\restr Y)$ is closed in
    $G_{T}\restr Y$.
  \end{prop}
  \begin{proof}
    Since $\Sigma$ is an equivalence relation, 0 belongs to $H(T)$, and
    $g\in H(T)$ implies $-g\in H(T)$. Suppose
    that $m,n \in H(T)$, say $m = p_1 - q_1$ and $n = p_2 - q_2$ with
    $(p_i, q_i) \in \Sigma$. Lemma~\ref{lem:Sigma properties} implies
    that $(p_1 + p_2, q_1 + q_2) \in \Sigma$, and therefore that $m+n
    = p_1 + p_2 - q_1 - q_2$ belongs to $H(T)$. So $H(T)$ is a
    subgroup of~$\Z^k$.

    Take $x \in Y$ and $g \in H(T)$. By Lemma~\ref{lem-pos-cone}, there
    exists $(p,q)\in\Sigma$ such that $g=p - q$. Choose an open
    neighbourhood $U$ of $x$ in $Y$ on which $T^p$ and $T^q$ are
    homeomorphisms. By choice of $Y$ we have $T^p y = T^q y$ for all
    $y \in U$, and hence $\sset{(y, g, y) : y \in U} = Z(U, p, q, U)$
    is an open neighbourhood of $(x, g, x)$ contained in $\sset{(y, g,
    y) : y \in Y, g \in H(T)}$. So $\sset{(y, g, y) : y \in Y, g \in
    H(T)} \subset \Iso(G_T)^\circ$. For the reverse inclusion,
    suppose that $(z,h,z) \in \Iso(G)^\circ$. By
    Lemma~\ref{lem-lch-top}, there exist $m,n \in \N^k$ and open sets
    $U,V \subset Y$ such that $(z, h, z) \in Z(U, m, n, V) \subset
    \Iso(G_T)$ with $T^m U = T^n V$.  So $T^m x = T^n x$ for all $x
    \in U$, and then $(m,n) \in \Sigma_U \subset \Sigma$. Thus $h \in
    H(T)$ as required.

    An application of \cite{kps:xx14}*{Proposition~2.1} to the $\Z^k$-valued
    cocycle $c : (x,g,x) \mapsto g$ on $G_{T}\restr Y$ shows that $\intiso(G_{T}\restr
    Y)$ is closed.
  \end{proof}

\begin{remark}
We have an opportunity to fill a gap in the literature. The penultimate paragraph of the
proof of the proof of \cite{ckss:fja14}*{Theorem~5.3}, appeals to
\cite{ckss:fja14}*{Corollary~2.8}. But unfortunately, the authors forgot to verify the
hypothesis of \cite{ckss:fja14}*{Corollary~2.8} that $\Gamma$ should be aperiodic. We
rectify this using our results above. Using the definition of aperiodicity of $\Gamma$
\cite{ckss:fja14}*{page~2575} and of the groupoid $G_\Gamma$ of $\Gamma$
\cite{ckss:fja14}*{page~2573} as in the proof of \cite{ckss:fja14}*{Corollary~2.8}, we
see that $\Gamma$ is aperiodic if and only if $G_\Gamma$ is topologically principal. In
the situation of \cite{ckss:fja14}*{Theorem~5.3}, the groupoid $G_{H\Lambda T}$ discussed
there is the restriction of the Deaconu--Renault groupoid $G_{\Lambda T}$ to $Y =
H\Lambda^\infty$ which has the properties required of $Y$ in
Proposition~\ref{prop-iso-on-y-closed} (see \cite{ckss:fja14}*{Theorem~4.2(2)}), and so
Proposition~\ref{prop-iso-on-y-closed} shows that $\intiso(G_{H\Lambda T})$ is closed. It
is easy to check that $G_\Gamma$ is isomorphic to $G_{H\Lambda T}/\intiso(G_{H\Lambda
T})$. So Proposition~\ref{prop-quotient-map}(\ref{it:quotient TP}) shows that $G_\Gamma$
is topologically principal and hence that $\Gamma$ is aperiodic as required.
\end{remark}

  \begin{cor}
    \label{cor-res-to-y}
    Let $T$ be an irreducible action of $\N^k$ by local homeomorphisms
    on a locally compact Hausdorff space $X$. Let $\Sigma$ and $H(T)$
    be as in \eqref{eq:5}~and~\eqref{eq:H(T),Ymax}.  Suppose that $Y$
    is an open subset of $X$ such that $T^{p}Y\subset Y$ for all $p$
    and such that $\Sigma_{Y}=\Sigma$.
    \begin{enumerate}
    \item\label{it:iotamap} Regard $C_{c}(G_{T}\restr Y)$ as a subalgebra of
        $C_{c}(G_{T})$. The identity map extends to a monomorphism $\iota :
        \cs(G_{T}\restr Y) \to \cs(G_{T})$, and $\iota(\cs(G_{T}\restr Y))$ is a
        hereditary subalgebra of $\cs(G_T)$.
    \item\label{it:circiota} The map $\pi\mapsto \pi\circ \iota$ is a bijection from
        the collection of irreducible representations of $\cs(G_{T})$ that are
        injective on $C_{0}(X)$ to the space of irreducible representations of
        $\cs(G_{T}\restr Y)$ that are injective on $C_{0}(Y)$.  Moreover, the map
        $\ker\pi\mapsto \ker(\pi\circ\iota)$ is a homeomorphism from
        $\bigl\{I\in\Prim \cs(G_{T}):I\cap C_{0}(X)=\sset0\bigr\}$ onto $
        \bigl\{J\in\Prim \cs(G_{T}\restr Y):J\cap C_{0}(Y)=\sset0\bigr\}$.
    \end{enumerate}
  \end{cor}
  \begin{proof}
    The inclusion $C_c(G_T|_Y) \hookrightarrow C_c(G_T)$ is a
    $*$-homomorphism and continuous in the inductive-limit topology.
    Hence we get a homomorphism $\iota$.  Fix $x\in Y$.  Let $L^{x}$
    be the regular representation of $\cs(G_{T})$ on
    $\ell^{2}((G_{T})_{x})$.  Then $L^{x}\circ \iota$ leaves the
    subspace $\ell^{2}{\sset{(y,g,x) \in G_T :y\in Y}})$
    invariant. Hence $L^{x}\circ \iota$ is equivalent to
    $L_{Y}^{x}\oplus 0$ where $L_{Y}^{x}$ is the corresponding regular
    representation of $\cs(G_{T}\restr Y)$.  Since $G_{T}$ and
    $G_{T}\restr Y$ are both amenable by Lemma~\ref{lem:amenable},
    $\iota$ is isometric and hence a monomorphism.

    Let $\sset{f_{i}}$ be an approximate identity for $C_{0}(Y)$. For
    $f\in C_{c}(G_{T})$ we have $f_{i}ff_{i}\in C_{c}(G_{T}\restr Y)$.
    Thus $\iota(\cs(G_{T}\restr Y))$ is the closure of
    $\bigcup_{i}f_{i}\cs(G_{T})f_{i}$.  It follows easily that the
    image of $\iota$ is a hereditary subalgebra of $\cs(G_{T})$ as
    claimed.

    If $\pi$ is an irreducible representation of $\cs(G_{T})$ that is
    injective on $C_{0}(X)$, then it does not vanish on the ideal
    $I_{Y}$ in $\cs(G_{T})$ generated by $C_{0}(Y)$.  Clearly,
    $\iota(\cs(G_{T}\restr Y))$ is Morita equivalent to $I_{Y}$, and restriction of
    representations implements Rieffel induction from $I_{Y}$ to
    $\iota(\cs(G_{T}\restr Y))$.  Since Rieffel induction between Morita
    equivalent \cs-algebras takes irreducibles to irreducibles
    (\cite{rw:morita}*{Corollary~3.32}) and since $\pi\restr {I_{Y}}$
    is irreducible (\cite{arv:invitation}*{Theorem~1.3.4}),
    $\pi\circ\iota$ is irreducible and clearly injective on
    $C_{0}(Y)$.  If $\rho$ is an irreducible representation of
    $\cs(G_{T}\restr Y)$, then it extends to an irreducible
    representation of $I_{Y}$.  Since $I_{Y}$ is an ideal, this
    representation extends to a (necessarily irreducible)
    representation $\pi$ of $\cs(G_{T})$ such that $\rho=\pi\circ
    \iota$.  The kernel of $\pi\restr {C_{0}(X)}$ is proper and has
    $\N^{k}$-invariant support.  Since $T$ acts irreducibly,
    $C_{0}(X)$ is $G_T$-simple, and so
    $\ker\bigl(\pi\restr {C_{0}(X)}\bigr) = \sset0$ and we
    obtain the required bijection.

    The remaining assertion follows from this bijection and the
    Rieffel correspondence (see \cite{rw:morita}*{Corollary~3.33(a)}).
  \end{proof}

  \section{The primitive ideals of the \texorpdfstring{$C^*$}{C*}-algebra of an irreducible
    Deaconu--Renault groupoid}
  \label{sec:prim-ideal-space}

  In this section we specialize to the situation where $T$ is an
  irreducible action of $\N^{k}$ on a locally compact Hausdorff space
  $Y$ with the property that, in the notation of~\eqref{eq:5},
  $\Sigma_{Y}=\Sigma$. We then have $\Sigma=\Sigma_{U}$ for all
  nonempty open subsets $U$ of $Y$ by Lemma~\ref{lem:SigmaU
    properties}.  Lemma~\ref{lem-pos-cone} says that $m-n \in H(T)$
  implies $T^{m}x=T^{n}x$ for all $x\in Y$; and
  Proposition~\ref{prop-iso-on-y-closed} gives
  \[
  \intiso(G_{T})=\sset{(x,g,x):\text{$x\in Y$ and $g\in H(T)$}}.
  \]
  We show that under these hypotheses, the primitive ideals of
  $\cs(G_{T})$ with trivial intersection with $C_{0}(Y)$ are indexed
  by characters of $H(T)$.  More precisely, we show that the
  irreducible representations of $\cs(G_{T})$ that are faithful on
  $C_{0}(Y)$ are indexed by pairs $(\pi,\chi)$ where $\pi$ is an
  irreducible representation of $\cs(G_{T}/\intiso(G_{T}))$ and $\chi$
  is a character of $H(T)$.  Our approach is to exhibit $\cs(G_{T})$
  as an induced algebra. Recall from
  Proposition~\ref{prop-iso-on-y-closed} that $\intiso(G_T)$ is closed
  in $G_T$, so Proposition~\ref{prop-quotient-map} gives a
  homomorphism $\kappa : \cs(G_T) \to \cs(G_T/\intiso(G_T))$.

  \begin{lemma}
    \label{lem:actions} Suppose that $T$ is an irreducible action of
    $\N^{k}$ on a locally compact space $Y$ such that
    $\Sigma_{Y}=\Sigma$.  There is an action $\alpha$ of $\T^{k}$ on
    $\cs(G_{T})$ such that $\alpha_{z}(f)(x,g,y)=z^{g}f(x,g,y)$ for
    $f\in \cc(G_{T})$.  Let $\qs:\cs(G_{T})\to
    \cs(G_{T}/\intiso(G_{T}))$ be the homomorphism of
    Proposition~\ref{prop-quotient-map}. There is an action $\alphat$
    of $H(T)^{\perp}$ on $\cs(G_{T}/\intiso(G_{T}))$ such that
    $\alphat_{z}\circ \qs = \qs\circ \alpha_{z}$ for all $z\in
    H(T)^{\perp}\subset \T^{k}$.

    If $\bar zw\notin H(T)^{\perp}$, then $\big(\ker(\qs\circ
    \alpha_{z}) + \ker(\qs\circ \alpha_{w})\big) \cap C_0(Y) \not=
    \{0\}$. We have $\ker(\qs\circ\alpha_{z}) = \ker(\qs\circ
    \alpha_{w})$ if and only if $\bar zw\in H(T)^{\perp}$.
  \end{lemma}
  \begin{proof}
    Let $c : G_T \to \Z^k$ be the canonical cocycle $c(x, g, y) = g$.
    The formula $\alpha_{z}(f)(\gamma)=z^{c(\gamma)}f(\gamma)$ defines
    a $*$-homomorphism $\alpha_z : \cc(G_{T}) \to \cc(G_T)$.  This
    $\alpha_z$ is trivially $I$-norm preserving, so extends to
    $\alpha_z : \cs(G_T) \to \cs(G_T)$. Since $\alpha_{\bar z}$ is an
    inverse for $\alpha_z$, we have $\alpha_z \in \Aut(\cs(G_T))$.
    The map $z \mapsto \alpha_z$ is a homomorphism because
    $\alpha_{zw}$ and $\alpha_z \circ \alpha_w$ agree on each
    $\cc(c^{-1}(g))$. To see that $z \mapsto \alpha_z$ is strongly
    continuous, first note that if $f \in C_c(G_T)$ is supported on
    $c^{-1}(g)$, then each $\alpha_z(f) = z^g f$, so $z \mapsto
    \alpha_z(f)$ is continuous. Since each $f \in C_c(G_T)$ is a
    finite linear combination $f = \sum_{\supp(f) \cap c^{-1}(g) \not=
      \emptyset} f|_{c^{-1}(g)}$ of such functions, $z \mapsto
    \alpha_z(f)$ is continuous for each $f \in C_c(G_T)$. Now an
    $\varepsilon/3$ argument shows that $z\mapsto\alpha_{z}$ is
    strongly continuous.

    Let $q:\Z^{k}\to \Z^{k}/H(T)$ be the quotient map. We have,
    \begin{equation*}
      \intiso(G_{T})=\sset{(x,g,x):\text{$x\in Y$ and $g\in H(T)$}}.
    \end{equation*}
    Identify $G_{T}/\intiso(G_{T})$ with $\sset{(x,q(g),y):(x,g,y)\in
      G_{T}} \subset Y\times (\Z^{k}/H(T)) \times Y$.

    Proposition~\ref{prop-quotient-groupoid} implies that the quotient
    map from $G_{T}$ onto $G_{T}/\intiso(G_{T})$ is continuous and
    open, so the sets
    \[
    \tZ(U,q(m),q(n),V) = \sset{(x,q(m-n),y):\text{$x\in U$, $y\in V$
        and $T^{m}x=T^{n}y$}}
    \]
    are a basis for the topology on $G_T/\intiso(G_T)$ (this makes
    sense because $T^{m}x=T^{n}y$ if and only if $T^{m+a}x=T^{n+b}y$
    whenever $a-b\in H(T)$).  Arguing as in the first paragraph, we
    get an action $\alphat$ of $H(T)^{\perp}$ on
    $\cs(G_{T}/\intiso(G_{T}))$ such that
    $\alphat_{z}(f)(x,q(g),y)=z^{g}f(x,q(g),y)$ for $f\in
    \cc(G_{T}/\intiso(G_{T}))$. For $f\in\cc(G_{T})$, it is easy to
    check that $\alphat_{z}\circ \qs(f)= \qs\circ \alpha_{z}(f)$ for
    $z \in H(T)^\perp$.  This identity then extends by continuity to
    all of $\cs(G_{T})$.

    Suppose that $\bar zw \notin H(T)^{\perp}$. Choose $n\in H(T)$
    such that $z^{n} \not= w^n$.  Fix a nonzero function $f\in \cc(Y)$
    and define $f_{n}\in \cc(\sset{(x,n,x):x\in Y}\subset \cc(G_{T})$
    by $f_{n}(x,n,x)=f(x,0,x)$ for all $x\in Y$. Then $w^n f-f_{n}\in
    \ker(\qs \circ \alpha_w)$ and $z^{n}f-f_{n}\in \ker(\qs\circ
    \alpha_{z})$.  Hence $(z^n - w^n)f \in \big(\ker(\qs \circ
    \alpha_z) + \ker(\qs\circ \alpha_{w})\big) \cap C_0(Y) \setminus
    \{0\}$ by choice of $n$. This proves the second-last statement of
    the lemma.

    Since each of $\qs\circ\alpha_z$ and $\qs\circ\alpha_w$ is
    injective on $C_0(Y)$, this also proves the (contrapositive of
    the) implication $\implies$ in the final statement of the
    lemma. For the reverse implication, suppose that $\bar zw\in
    H(T)^{\perp}$. Then
    \[
    \ker(\qs\circ \alpha_{w}) = \ker(\qs \circ \alpha_{\bar zw} \circ
    \alpha_z) = \ker(\alphat_{\bar zw} \circ \qs \circ \alpha_z) =
    \ker(\qs\circ\alpha_z).\qedhere
    \]
  \end{proof}

  The final assertion of Lemma~\ref{lem:actions} ensures that we can
  form the induced algebra
  $\Ind_{H(T)^{\perp}}^{\T^{k}}\bigl(\cs(G_{T}/\intiso(G_{T})),\alphat\bigr)$,
  namely
  \[
  \sset{s\in C(\T^{k},\cs(G_{T}/\intiso(G_{T})):
    \text{$s(wz)=\alphat_{z}(s(w))$ for all $w\in\T^{k}$ and $z\in
      H(T)^{\perp}$}}.
  \]
  Induced algebras have a well-understood structure.  Some of their
  elementary properties (in particular, the ones that we rely upon)
  are discussed in \cite{rw:morita}*{\S6.3}.

  Before proving the next result, we recall some basic
  results from abelian harmonic analysis.  We
  write $\cc(H(T))$ for the set of finitely supported functions on
  $H(T)$.  If $\phi\in \cc(H(T))$, then its Fourier transform
  $\hat\phi \in C(\T^k)$ is given by
  \begin{equation*}
    \hat\phi(z)=\sum_{n\in H(T)}\phi(n)z^{n},
  \end{equation*}
  and is constant on $H(T)^{\perp}$ cosets.  Taking a few liberties
  with notation and terminology, we regard $\hat \phi$ as an element
  of $C(\T^{k}/H(T)^{\perp})$.  The general theory implies that
  $\sset{\hat\phi:\phi\in\cc(H(T))}$ is a (uniformly) dense subalgebra
  of $C(\T^{k}/H(T)^{\perp})$.

  \begin{lemma}
    \label{lem-Ht-closed}
    Let $T$ be an irreducible action of $\N^{k}$ on a locally compact
    space $Y$ by local homeomorphisms, and suppose that
    $\Sigma_{Y}=\Sigma$.  If $(x,g,y)\in G_{T}$, then $(x,g+n,y) \in
    G_T$ for all $n\in H(T)$.
  \end{lemma}
  \begin{proof}
    Let $(x,g,y)=(x,p-q,y)$ with $T^{p}x=T^{q}y$.  Fix $n\in
    H(T)$. Then $n=n_+-n_-$ with $(n_+,n_-)\in \Sigma=\Sigma_{Y}$.
    Hence $T^{n_+}z=T^{n_-}z$ for all $z\in Y$, giving
    \begin{equation*}
      T^{p+n_+}x=T^{n_+}T^{p}x=T^{n_+}T^{q}y=T^{n_-}T^{q}y=T^{q+n_-}y.
    \end{equation*}
    Hence $(x,g+n,y) = (x,(p+n_+)-(q+n_-),y) \in G_{T}$.
  \end{proof}

  Because of Lemma~\ref{lem-Ht-closed}, we can define a left action of
  $\cc(H(T))$ on $\cc(G_{T})$ by
  \begin{equation}
    \label{eq:7}
    \phi\cdot f(x,g,y):=\sum_{n\in H(T)} \phi(n) f(x,g-n,y).
  \end{equation}

  \begin{lemma}
    \label{lem-cts-linear}
    Let $T$ be an irreducible action of $\N^{k}$ on a locally compact
    space $Y$ by local homeomorphisms such that $\Sigma_{Y}=\Sigma$,
    and let $\kappa : \cs(G_T) \to \cs(G_T/\intiso(G_T))$ be as in
    Proposition~\ref{prop-quotient-map}.  Then
    \begin{equation*}
      \qs(\alpha_{z}(\phi\cdot f))=\hat\phi(z)\qs(\alpha_{z}(f))
    \end{equation*}
    for all $f\in\cc(G_{T})$, all $z\in\T^{k}$, and all $\phi \in
    \cc(H(T))$.
  \end{lemma}
  \begin{proof}
    We compute:
    \begin{align*}
      \qs(\alpha_{z}(\phi\cdot f))(x,q(g),y) &= \sum_{m\in H(T)}
      \alpha_{z}(\phi\cdot f)(x,g+m,y) \\
      &= \sum_{m\in H(T)} z^{g+m} \phi\cdot f(x,g+m,y) \\
      &= \sum_{m\in H(T)}\sum_{n\in H(T)}z^{g+m}\phi(n) f(x,g+m-n,y).
      \intertext{Since both sums are finite and we can interchange the
        order of summations at will, we may continue the calculation:}
      &=  \sum_{m\in H(T)}\sum_{n\in H(T)} z^{g+m+n}\phi(n) f(x,g+m,y) \\
      &=
      \sum_{m\in H(T)} z^{g+m} \hat\phi(z) f(x,g+m,y) \\
      &= \hat\phi(z) \qs(\alpha_{z}(f))(x,q(g),y).\qedhere
    \end{align*}
  \end{proof}

  \begin{prop}
    \label{prp:induced alg}
    Let $T$ be an irreducible action of $\N^{k}$ on a locally compact
    space $Y$ by local homeomorphisms, and suppose that
    $\Sigma_{Y}=\Sigma$.  Let
    \[
    \alpha:\T^{k}\to \Aut\cs(G_{T})
    \quad\text{ and }\quad
    \tilde\alpha : H(T)^\perp \to \Aut\cs(G_T/\intiso(G_T))
    \]
    be as in Lemma~\ref{lem:actions}, and let $\kappa : \cs(G_T) \to
    \cs(G_T/\intiso(G_T))$ be as in
    Proposition~\ref{prop-quotient-map}. There is an isomorphism $\Phi
    : \cs(G_T) \to \Ind_{H(T)^{\perp}}^{\T^{k}}
    (\cs(G_{T}/\intiso(G_{T})), \alphat)$ such that
    $\Phi(a)(z)=\qs(\alpha_{z}(a))$ for $a \in \cs(G_T)$ and all $z
    \in \T$.
  \end{prop}
  \begin{proof}
    For $a \in \cs(G_T)$, the map $z\mapsto \qs(\alpha_z(a))$
    is continuous by continuity of $\alpha$. Take $f\in \cc(G_{T})$, $w\in \T^{k}$ and
    $z\in H(T)^{\perp}$. Lemma~\ref{lem:actions} gives $\alphat_{z}\circ \qs=\qs\circ
    \alpha_{z}$. Hence
    \begin{equation*}
      \Phi(f)(wz)=\qs(\alpha_{wz}(f))=\qs(\alpha_{z}(\alpha_{w}(f)))=
      \alphat_{z}\qs(\alpha_{w}(f))=\alphat_{z}(\Phi(f)(w)) .
    \end{equation*}
    Thus $\Phi$ takes values in $\Ind_{H(T)^{\perp}}^{\T^{k}}
    (\cs(G_{T}/\intiso(G_{T})), \alphat)$. It is not hard to check
    that $\Phi$ is a homomorphism.

    To see that $\Phi$ is injective we use an averaging argument.  Let
    $\T^{k}$ act on the left of
    $\Ind_{H(T)^{\perp}}^{\T^{k}}(\cs(G_{T}/\intiso(G_{T})),\alphat)$
    by left translation: $\lt_{z}(c)(w)=c(\bar zw)$. We have
    $\Phi\circ \alpha_{z}= \lt_{\bar z}\circ \Phi$.  So the standard
    argument involving the faithful conditional expectations obtained
    from averaging over $\T^k$ actions (see, for example,
    \cite{sww:xx13}*{Lemma~3.13}) shows that it is sufficient to check
    that $\Phi$ restricts to an injection on $C^*(G_T)^\alpha$.

    If $f\in\cc(G_{T})$, then arguing as in \cite{wil:crossed}*{Lemma~1.108},
    we have $\int_{\T^{k}}\alpha_{z}(f)\,dz\in \cc(G_{T})$ and for
    $\gamma \in G_T$,
    \begin{align*}
    \Big(\int_{\T^{k}}\alpha_{z}(f)\,dz\Big)(\gamma) &=\int_{\T^{k}}
    \alpha_{z}(f)(\gamma)\,dz = \Bigl(\int_{\T^{k}}
    z^{c(\gamma)}\,dz\Bigr) f(\gamma) \\
& = \begin{cases}
      f(\gamma)&\text{if $\gamma\in c^{-1}(0)$}\\
      0&\text{otherwise.}
    \end{cases}
    \end{align*}
    It follows that $\cs(G_{T})^{\alpha} = \overline{\cc(c^{-1}(0))}
    \subset \cs(G_{T})$.  Thus the inclusion map induces a
    monomorphism $\rho:\cs(c^{-1}(0))\to\cs(G_{T})$ whose image is
    exactly $\cs(G_{T})^{\alpha}$.  To see that
    $\Phi\restr{\cs(G)^{\alpha}}$ is injective, it suffices to show
    that $\Phi\circ\rho$ is injective.  Since $c^{-1}(0)$ is amenable
    by Lemma~\ref{lem:amenable} and principal by construction,
    \cite{exe:pams11}*{Theorem~4.4} implies that we need only show
    that $(\Phi\circ \rho)\restr{C_{0}(Y)}$ is injective.  As $\rho$
    restricts to the canonical inclusion $C_{0}(Y) \hookrightarrow
    \cs(G_{T})^{\alpha}$, it is enough to verify that $\Phi$ is
    injective on $C_{0}(Y)$.  The homomorphism $\qs\circ \alpha_{z}$
    restricts to the identity map of $C_{0}(Y)\subset \cs(G_{T})$ onto
    $C_{0}(Y)\subset \cs(G_{T}/\intiso(G_{T}))$.  So if $f\in
    C_{0}(Y)$, $z\in\T^{k}$ and $b\in G_{T}/\intiso(G_{T})$, then
    \begin{equation*}
      \Phi(f)(z)(b)=\qs(\alpha_{z}(f))(b)=
      \begin{cases}
        f(x)&\text{if $b=(x,0,x)\in (G_{T}/\intiso(G_{T}))^{(0)}$} \\
        0&\text{otherwise.}
      \end{cases}
    \end{equation*}
    Thus if $\Phi(f)=0$, then $f=0$.  This completes the proof that
    $\Phi$ is injective.

    We still have to show that $\Phi$ is
    surjective. Lemma~\ref{lem-cts-linear} implies that if $\phi\in
    \cc(H(T))$ and $f\in \cc(G_{T})$, then $\hat\phi\cdot
    \Phi(f)=\Phi(\phi\cdot f)$ for the obvious action of
    $C(\T^{k}/H(T))$ on the induced algebra.  Since $\sset{\hat\phi :
      \phi \in \cc(H(T))}$ is dense in $C(\T^{k}/H(T))$ it follows
    that the range of $\Phi$ is a $C(\T^{k}/H(T))$-submodule.  So it
    suffices to show that the range of $\qs\circ\alpha_{z}$ contains
    $\cc(G_{T}/\intiso(G_{T}))$.

    For this, fix $g \in \Z^k$ and $f \in \tilde{c}^{-1}(q(g))$; it
    suffices to show that $f$ is in the range of $\pi \circ
    \alpha_z$. Define $h \in C_c(G_T)$ by
    \begin{equation*}
      h(\gamma) = \begin{cases}
        \overline{z}^g f(\tilde{q}(\gamma)) &\text{ if $c(\gamma) = g$}\\
        0 &\text{ otherwise.}
      \end{cases}
    \end{equation*}
    Then $h$ is continuous because each $c^{-1}(g)$ is clopen in
    $G_T$; and $\qs(\alpha_z(h)) = f$.
  \end{proof}

    We now aim to apply \cite{rw:morita}*{Proposition~6.6}, which
    describes the primitive-ideal space of an induced algebra, to
    describe the topology of $\Prim(\cs(G_T))$ for a special class of
    $\N^k$-actions $T$. To achieve this we first describe, in
    Lemma~\ref{lem:Jacobson}, the Jacobson topology on
    $\Prim(\cs(G))$ when $G$ is an amenable \'etale Hausdorff
    groupoid whose reduction to any closed invariant set is
    topologically principal. This topology is also described by
    \cite{ren:jot91}*{Corollary~4.9}, but the statement given there is
    not quite the one we need.

\begin{lemma}\label{lem:augmentations irreducible}
  Let $G$ be a second-countable locally compact Hausdorff \'etale
  groupoid, and fix $x \in \go$. There is an irreducible
  representation $\omega_{[x]} : C^*(G) \to \Bb(\ell^2([x]))$
  satisfying $\omega_{[x]}(f) \delta_y = \sum_{s(\gamma) = y}
  f(\gamma) \delta_{r(\gamma)}$ for all $f \in C_c(G)$. If $G$ is
  topologically principal and amenable and if $[x]$ is dense in $\go$,
  then $\omega_{[x]}$ is faithful, and hence $C^*(G)$ is primitive.
\end{lemma}
\begin{proof}
  Let $E_x$ denote the 1-dimensional representation of the group
  $G^x_x$. Then $\omega_{[x]} := \Ind^G_{\{x\}} E_x$ is a
  representation satisfying the desired formula.\footnote{This is also
  the representation described in~\cite{bcfs:sf14}*{Proposition~5.2}.}
  Hence $\omega_{[x]}$ is irreducible by \cite{ionwil:pams08}*{Theorem~5}.

  Now suppose that $G$ is amenable and topologically principal with
  $[x]$ dense in $\go$. Then clearly $\omega_{[x]}$ is faithful on
  $C_0(\go)$. So \cite{exe:pams11}*{Theorem~4.4} says that it is
  faithful on $C^*(G)$, whence $C^*(G)$ is primitive.
\end{proof}

Recall that the quasi-orbit space $\qo(G)=\set{\overline{[x]}:x\in\go}$ carries the
quotient topology for the map $q:\go\to \qo(G)$ that identifies $u$ with $v$ exactly when
$[u]$ and $[v]$ have the same closure in $\go$.  In particular, if $S\subset \qo(G)$,
then $\overline{S}=\set{q(x):x\in \overline{q^{-1}(S)}}$.

\begin{lemma}\label{lem:Jacobson}
  Let $G$ be an amenable, \'etale Hausdorff groupoid and suppose that
  $G|_X$ is topologically principal for every closed invariant subset
  $X$ of the unit space\footnote{Although the term has been used
  inconsistently, in \cite{ren:jot91} for example, one says the
  $G$-action on $\go$ is essentially free.} For $x \in \go$, let
  $\omega_x$ be the irreducible representation of
  Lemma~\ref{lem:augmentations irreducible}. The map $x \mapsto
  \ker\omega_x$ from $\go$ to $\Prim(\cs(G))$ descends to a
  homeomorphism of the quasi-orbit space $\qo(G)$ onto
  $\Prim(\cs(G))$.
\end{lemma}
\begin{proof}
  For $x \in \go$, we have $\ker\omega_x \cap C_0(\go) = C_0(\go
  \setminus \overline{[x]})$. Since $G|_X$ is topologically principal
  for every closed invariant subset $X \subset \go$,
  \cite{ren:jot91}*{Corollary~4.9}\footnote{Specifically,
    \cite{ren:jot91}*{Corollary~4.9} applied to the groupoid dynamical
    system $(G, \Sigma, \Aa)$ where $\Sigma$ is the bundle of trivial
    groups over $\go$ and $\Aa$ is the trivial bundle $\go \times \C$ of
    1-dimensional $C^*$-algebras---see also
    \cite{bcfs:sf14}*{Corollary~5.9}.}
  therefore implies that $\ker\omega_x = \ker\omega_y$ if and only if
  $\overline{[x]} = \overline{[y]}$.  Hence $x \mapsto \ker\omega_x$
  descends to a well-defined injection $\overline{[x]} \mapsto
  \ker\omega_x$. To see that it is surjective, observe that if $\pi$
  is an irreducible representation of $\cs(G)$, then
  Proposition~\ref{prop-factors} implies that $\ker\pi \cap C_0(\go) =
  C_0(\go\setminus\overline{[x]})$ for some $x \in \go$. That is,
  $\ker\pi \cap C_0(\go) = \ker\omega_x \cap C_0(\go)$, and then
  \cite{ren:jot91}*{Corollary~4.9} again shows that $\ker\pi =
  \ker\omega_x$.

  To show that $\overline{[x]} \mapsto \ker\omega_x$ is a
  homeomorphism, it suffices to take a set $S \subset \qo(G)$ and an
  element $x \in \go$ and show that $\overline{[x]} \in \overline{S}$
  if and only if $\ker\omega_x \in \overline{\sset{\ker\omega_y : q(y)
      \in S}}$; for then $S \subset \qo(G)$ is closed if and
  only if its image $\sset{\ker\omega_y : q(y) \in S}$ is
  closed in $\Prim(C^*(G))$.

  Fix $S \subset \qo(G)$ and $x \in \go$. We have
  \[
  \overline{\set{\ker\omega_y : q(y) \in  S}} =
  \sset{\ker\omega_z : \bigcap_{q(y) \in S} \ker\omega_y \subset
    \ker\omega_z}.
  \]
  Using \cite{ren:jot91}*{Corollary~4.9} again, we deduce that
  $\ker\omega_x \in \overline{\set{\ker\omega_y : q(y) \in S}}$
  if and only if $\big(\bigcap_{q(y) \in S} \ker\omega_y\big)
  \cap C_0(\go) \subset \ker\omega_x \cap C_0(\go)$. We have
  \begin{align*}
  \Big(\bigcap_{q(y) \in S} \ker\omega_y\Big) \cap C_0(\go) &=
  \set{f \in C_0(\go) : f|_{q^{-1}(S)} = 0} \\ &= \set{f \in C_0(\go) :
    f|_{\overline{q^{-1}(S)}} = 0} .
  \end{align*}
On the other hand,
\begin{equation*}
  \ker\omega_x \cap C_0(\go) = \set{f \in C_0(\go) :
    f|_{\overline{[x]}} = 0} .
\end{equation*}
Hence $\ker\omega_x \in
  \overline{\sset{\ker\omega_y : y \in \bigcup S}}$ if and only if
  $\overline{[x]} \subset \overline{q^{-1}(S)}$, and this is
  equivalent to $q(x) \in \overline{S}=\set{q(x):x\in
    \overline{q^{-1}(S)}}$ since $\overline{q^{-1}(S)}$ is
  closed and invariant.
\end{proof}

For the next result, recall that the quotient map $q : G \to
G/\intiso(G)$ restricts to a homeomorphism of unit spaces. Since $q$
also preserves the range and source maps, it carries $G$-orbits
bijectively to the corresponding $(G/\intiso(G))$-orbits, and
therefore carries orbit closures in $G$ to the corresponding orbit
closures in $G/\intiso(G)$. Hence the identification $\go =
G/\intiso(G)$ induces a homeomorphism $\qo(G) \cong
\qo(G/\intiso(G))$.

\begin{thm}\label{thm:specialcaseJacobson}
  Let $T$ be an irreducible action of $\N^{k}$ on a locally compact
  space $Y$ by local homeomorphisms such that, in the notation
  of~\eqref{eq:5}, $\Sigma_Y = \Sigma$. Suppose that for every $y \in
  Y$, the set
  \[
  \Sigma_{\overline{[y]}} :=
  \sset{(m,n)\in\N^{k}\times\N^{k}:\text{$T^{m}x=T^{n}x$ for all $x\in
      \overline{[y]}$}}
  \]
  satisfies $\Sigma_{\overline{[y]}} = \Sigma$. Let $\alpha :
  \T^{k}\to \Aut\cs(G_{T})$ be as in Lemma~\ref{lem:actions}, and let
  $\kappa : \cs(G_T) \to \cs(G_T/\intiso(G_T))$ be as in
  Proposition~\ref{prop-quotient-map}. For $y \in
  (G_T/\intiso(G_T))^{(0)}$, let $\omega_x$ be the irreducible
  representation of $C^*(G_T)$ described in
  Lemma~\ref{lem:augmentations irreducible}. The map $(y, z) \mapsto \ker(\omega_y \circ
  \alpha_z)$ from $Y \times \T^k$ to $\Prim(\cs(G_T))$ descends to a
  homeomorphism $\qo(G_T) \times {H(T)}^{\wedge} \cong
  \Prim(\cs(G_T))$.
\end{thm}
\begin{proof}
  Let $\Phi : \cs(G_T) \to \Ind_{H(T)^{\perp}}^{\T^{k}}
  (\cs(G_{T}/\intiso(G_{T})), \alphat)$ be the isomorphism of
  Proposition~\ref{prp:induced alg}. For each $y \in Y$, let
  $\tilde\omega_y$ be the irreducible representation of
  $\cs(G_T/\intiso(G_T))$ obtained from
  Lemma~\ref{lem:augmentations irreducible}. Observe that $\tilde\omega_y \circ \qs =
  \omega_y$. We have
  \[
  \Phi(\ker(\omega_y \circ \alpha_z)) = \sset{s \in
    \Ind_{H(T)^{\perp}}^{\T^{k}} (\cs(G_{T}/\intiso(G_{T})), \alphat)
    : f(z) \in \ker\tilde\omega_y}.
  \]
  Write $\varepsilon_z$ for the homomorphism of the induced
  algebra onto $\cs(G_{T}/\intiso(G_{T}))$ given by evaluation at $z$.
  It now suffices to show that
  \begin{equation}\label{eq:yz to irrep}
    (y,z) \mapsto \ker\bigl( \tilde\omega_y \circ \varepsilon_z \bigr)
  \end{equation}
  induces a homeomorphism of $\qo(G_T) \times {H(T)}^{\wedge}$ onto the
  primitive ideal space of the induced algebra.

  Proposition~\ref{prop-iso-on-y-closed} combined with the hypothesis
  that each $\Sigma_{\overline{[y]}} = \Sigma$ ensures that
  $\intiso(G)|_{\overline{[y]}} = \intiso(G|_{\overline{[y]}})$ for
  each $y$. Hence Proposition~\ref{prop-quotient-groupoid} ensures
  that the reduction of $G_{T}/\intiso(G_{T})$ to any orbit closure,
  and hence to any closed invariant set, is topologically
  principal. Now Lemma~\ref{lem:Jacobson} implies that
  $\ker(\tilde\omega_y \circ \varepsilon_z) = \ker(\tilde\omega_x
  \circ \varepsilon_z)$ if and only if $\overline{[y]} =
  \overline{[x]}$. So the map~\eqref{eq:yz to irrep} descends to a map
  $(\overline{[y]}, z) \mapsto \ker\bigl(\tilde\omega_y \circ
  \varepsilon_z\bigr)$. Composing this with the homeomorphism of
  Lemma~\ref{lem:Jacobson} shows that~\eqref{eq:yz to irrep} induces a
  well-defined map
  \[
  M : (\ker\tilde\pi_\omega, z) \mapsto \ker\bigl(\tilde\omega_y \circ
  \varepsilon_z\bigr).
  \]
  An application of Proposition~6.16 of \cite{rw:morita}---or, rather, of
  the obvious primitive-ideal version of that result---shows that $M$
  induces a homeomorphism of the quotient of
  $(\Prim(\cs(G_{T}/\intiso(G_{T}))) \times \T^k)$ by the
  diagonal action of $H(T)^\perp$ onto the primitive ideal space of the induced
  algebra. Since the action of $H(T)^\perp$ on $\T^k$ is by
  translation and has quotient ${H(T)}^{\wedge}$, it now suffices to
  show that the action of $H(T)^\perp$ on
  $\Prim\bigl(\cs(G_{T}/\intiso(G_{T}))\bigr)$ is trivial. Since
  $\alphat_z$ fixes $C_0(\go) \subset
  \cs(G_{T}/\intiso(G_{T}))$ pointwise, for any ideal $I$ of
  $\cs(G_{T}/\intiso(G_{T}))$, we have $\alphat_z(I) \cap C_0(\go) = I
  \cap C_0(\go)$, and then \cite{ren:jot91}*{Corollary~4.9} implies
  that $\alphat_z(I) = I$. So $H(T)^\perp$ acts trivially on
  $\Prim\bigl(\cs(G_{T}/\intiso(G_{T}))\bigr)$.
\end{proof}

\section{The Primitive Ideals of the \texorpdfstring{$C^*$}{C*}-algebra of a Deaconu--Renault groupoid}
\label{sec:prim-ideals-csg_t}

In this section, our aim is to catalogue the primitive ideals of $\cs(G_{T})$.  We need
to refine our notation from Section~\ref{sec:prim-ideal-space} to accommodate actions
which are not necessarily irreducible.

\begin{notation}
  Let $T$ be an action of $\N^{k}$ on a locally compact space $X$ by
  local homeomorphisms.  Recall that for $x \in X$,
  \[
  [x]=\sset{y\in X:\text{$T^{m}x=T^{n}y$ for some $m,n\in\N^{k}$}}.
  \]
  For $x \in X$ and $U\subset \overline{[x]}$ relatively open, let
  \[
  \Sigma(x)_{U} := \sset{(m,n)\in\N^{k}\times\N^{k}:
    \text{$T^{m}y=T^{n}y$ for all $y\in U$}},
  \]
  and define
  \[\textstyle
  \Sigma(x) := \bigcup_{U}\Sigma(x)_{U}.
  \]
  Lemma~\ref{lem-pos-cone} implies that
  \begin{equation*}
    Y(x):=\bigcup\sset{Y\subset \overline{[x]}:\text{$Y$ is relatively
        open and $\Sigma(x)_{Y}=\Sigma(x)$}}
  \end{equation*}
  is nonempty and is the maximal relatively open subset of
  $\overline{[x]}$ such that $\Sigma(x)_{Y(x)}=\Sigma(x)$.
  Proposition~\ref{prop-iso-on-y-closed} implies that
  \begin{equation*}
    H(x) := H(T|_{\overline{[x]}}) = \sset{m-n:(m,n)\in\Sigma(x)}
  \end{equation*}
  is a subgroup of $\Z^{k}$.  To lighten notation, set $\II(x) :=
  \intiso(G_{T}\restr{Y(x)})$.  Proposition~\ref{prop-iso-on-y-closed}
  says that
  \begin{equation*}
    \II(x)=\sset{(y,g,y):\text{$y\in Y(x)$ and $g\in H(x)$}},
  \end{equation*}
  and is a closed subset of $G_{T}\restr{Y(x)}$.
\end{notation}

\begin{lemma}
  \label{lem-y-separate}
  Let $T$ be an action of $\N^{k}$ on a locally compact Hausdorff
  space $X$ by local homeomorphisms.  For $x,y\in X$, we have
  $Y(x)=Y(y)$ if and only if $\overline{[x]}=\overline{[y]}$.
\end{lemma}
\begin{proof}
  The ``if'' direction is trivial.  Suppose that $Y(x)=Y(y)$.  By
  symmetry, it suffices to show that $y\in\overline{[x]}$. Since
  $Y(x)=Y(y)$ is open in $\overline{[y]}$, we have $Y(x)\cap
  [y]\not=\emptyset$.  Since $Y(x)\subset [x]$, and $[x]$ is
  $G_{T}$-invariant, we deduce that $y\in \overline{[x]}$.
\end{proof}

The key to the proof of our main theorem is the following result,
which works at the level of irreducible representations.

\begin{thm}
  \label{thm:all irreps}
  Let $T$ be an action of $\N^{k}$ on a locally compact Hausdorff
  space $X$ by local homeomorphisms. Take $x \in X$ and $z
  \in \T^k$. Suppose that $\rho$ is a faithful irreducible representation
  of $C^*(G_{T|_{Y(x)}}/\Ii(x))$. Let $\iota : C^*(G_{T|_{Y(x)}}) \to
  C^*(G_T)$ be the inclusion of Corollary~\ref{cor-res-to-y}. Let
  \begin{equation*}
    \Phi : C^*(G_{T|_{Y(x)}}) \to
    \Ind_{H(x)^\perp}^{\T^k}(C^*(G_{T|_{Y(x)}}/\Ii(x)), \tilde{\alpha})
  \end{equation*}
  be the isomorphism of Proposition~\ref{prp:induced alg}, and let
  \begin{equation*}
    \varepsilon_z :
    \Ind_{H(x)^\perp}^{\T^k}\big(C^*(G_{T|_{Y(x)}}/\Ii(x)),
    \tilde{\alpha}\big) \to
    C^*(G_{T|_{Y(x)}}/\Ii(x))
  \end{equation*}
  denote evaluation at $z$. Let $R_x : C^*(G_T) \to
  C^*(G_{T|_{\overline{[x]}}})$ be the homomorphism induced by
  restriction of compactly supported functions. There is a unique
  irreducible representation $\pi_{x, z, \rho}$ of $C^*(G_T)$ such
  that
  \begin{enumerate}
  \item\label{it:factors} $\pi_{x, z, \rho}$ factors through $R_x$, and
  \item\label{it:agrees} the representation $\pi^0_{x, z, \rho}$ of
      $C^*(G_{T|_{\overline{[x]}}})$ such that $\pi_{x, z, \rho} = \pi^0_{x, z, \rho}
      \circ R_x$ satisfies $\pi^0_{x,z,\rho} \circ \iota = \rho \circ \varepsilon_z
      \circ \Phi$.
  \end{enumerate}
  Every irreducible representation of $C^*(G_T)$ has the form $\pi_{x, z,
  \rho}$ for some $x$, $z$, $\rho$.
\end{thm}
\begin{proof}
  The representation $\rho \circ \varepsilon_z \circ \Phi$ is an
  irreducible representation of $C^*(G_{T|_{Y(x)}})$, and is injective
  on $C_0(Y(x))$ because both $\Phi$ and $\varepsilon_z$ restrict to
  injections on $C_0(Y(x))$.
  Corollary~\ref{cor-res-to-y}(\ref{it:circiota}) applied to
  $Y(x)\subset \overline{[x]}$ yields a unique
  representation $\pi^0_{x,z,\rho}$ of $C^*(G_{T|_{\overline{[x]}}})$
  such that $\pi^0_{x,z,\rho} \circ \iota = \rho \circ \varepsilon_z
  \circ \Phi$.  The set $\overline{[x]}$ is a closed invariant set in
  $X$.  As in Proposition~\ref{prop-factors}, restriction of functions
  induces a homomorphism $R_x : C^*(G_T) \to
  C^*(G_{T|_{\overline{[x]}}})$. Now $\pi_{x, z, \rho} := \pi^0_{x, z,
    \rho} \circ R_x$ satisfies (\ref{it:factors})~and~(\ref{it:agrees}).

  For uniqueness, take a representation $\phi$ of
  $C^*(G_T)$ satisfying (\ref{it:factors})~and~(\ref{it:agrees}).
  Then $\phi$ vanishes on the ideal
  generated by $C_0(X \setminus \overline{[x]})$ which is precisely
  the kernel of $R_x$ by Proposition~\ref{prop-factors}. So $\phi =
  \phi_0 \circ R_x$ for some irreducible representation $\phi_0$ of
  $C^*(G_{T|_{\overline{[x]}}})$ satisfying $\phi_0 \circ \iota = \rho
  \circ \varepsilon_z \circ \Phi$. We saw in the preceding paragraph
  that $\pi^0_{x, z, \rho}$ is the unique such representation, so
  $\phi^0 = \pi^0_{x, z, \rho}$ and hence $\phi = \pi_{x, z, \rho}$.

  To see that every irreducible representation of $C^*(G_T)$ has the
  form $\pi_{x, z, \rho}$, fix an irreducible representation $\phi$ of
  $C^*(G_T)$. Since it is irreducible, Proposition~\ref{prop-factors}
  implies that $\phi = \phi^0 \circ R_x$ for some $x \in X$ and some
  irreducible representation $\phi^0$ of
  $C^*(G_{T|_{\overline{[x]}}})$ that is faithful on
  $C_0(\overline{[x]})$. Since $\Phi$ is an isomorphism,
  Corollary~\ref{cor-res-to-y}(\ref{it:circiota}) implies that
  $\phi^0$ is uniquely determined by $\phi^0 \circ \iota \circ
  \Phi^{-1}$, which is an irreducible representation of
  $\Ind_{H(x)^\perp}^{\T^k}(C^*(G_{T|_{Y(x)}}/\Ii(x)),
  \tilde{\alpha})$ that is faithful on $C_0(Y(x))$. By
  \cite{rw:morita}*{Proposition~6.16}, there exists $z$ such that
  $\ker(\varepsilon_z) \subset \ker \phi^0 \circ \iota \circ
  \Phi^{-1}$, and then $\phi^0 \circ \iota \circ \Phi^{-1}$ descends
  to an irreducible representation $\rho$ of
  $C^*(G_{T|_{Y(x)}}/\Ii(x))$. That is $\rho \circ \varepsilon_z =
  \phi^0 \circ \iota \circ \Phi^{-1}$. Post-composing with $\Phi$ on
  both sides of this equation shows that $\phi^0 \circ \iota = \rho
  \circ \varepsilon_z \circ \Phi$. So we now need only prove that $\rho$ is faithful.

  Since $\phi^0$ is faithful on $C_0(\overline{[x]})$, the composition
  $\phi^0 \circ \iota \circ \Phi^{-1}$ is faithful on $C_0(Y(x))$, and
  hence $\rho$ is faithful on $C_0(Y(x)) =
  C_0\big((G_{T|_{Y(x)}}/\Ii(x))^{(0)}\big)$.
  Proposition~\ref{prop-quotient-groupoid}(\ref{it:quotient TP})
  implies that $G_{T|_{Y(x)}}/\Ii(x)$ is topologically principal, and
  Proposition~\ref{prop-quotient-groupoid}(\ref{it:quotient amenable})
  combined with Lemma~\ref{lem:amenable} implies that
  $G_{T|_{Y(x)}}/\Ii(x)$ is amenable. So
  \cite{exe:pams11}*{Theorem~4.4} implies that $\rho$ is faithful as
  claimed.
\end{proof}

\begin{proof}[Proof of Theorem~\ref{thm-mainthm}]
  Fix $x \in \go_T$ and $z \in \T^k$. Let $\alpha_z \in
  \Aut(C^*(G_T))$ be the automorphism of Lemma~\ref{lem:actions}, and
  let $\omega_{[x]}$ be the irreducible representation of
  Lemma~\ref{lem:augmentations irreducible}. Then $\pi_{x, z} :=
  \omega_{[x]} \circ \alpha_z$ is an irreducible representation
  satisfying~\eqref{eq:pi_xz formula}. Furthermore
  $\pi_{x,z}|_{C_{0}(\go)}$ has support $\overline{[x]}$.

  It is clear that the relation $\sim$ is an equivalence relation.
  To see that $\ker \pi_{x,z} = \ker \pi_{y, w}$ if and only if
  $\overline{[x]} = \overline{[y]}$ and $\overline{z}w \in
  H(x)^\perp$, first suppose that $\overline{[x]} \not=
  \overline{[w]}$. Then $\ker\pi_{x,z} \cap C_0(X) \not= \ker\pi_{y,
    w} \cap C_0(X)$.

  Second, suppose that $\overline{[x]} = \overline{[y]}$ but
  $\overline{z}w \not\in H(x)$.  Then $\pi_{x,z}$ and $\pi_{y,w}$
  descend to representations $\pi^0_{x,z}$ and $\pi^0_{y,w}$ of
  $C^*(G_{T|_{\overline{[x]}}})$. Corollary~\ref{cor-res-to-y}(\ref{it:circiota})
  implies that their kernels are equal if and only if the kernels of
  $\pi^0_{x,z} \circ \iota$ and $\pi^0_{y, w} \circ \iota$ are
  equal. Lemma~\ref{lem-y-separate} shows that $Y(x) = Y(y)$, and for
  $f\in\cc(G_{T}\restr{Y(x)}) = \cc(G_{T}\restr{Y(y)})$, we have
  \begin{equation*}
    \pi^{0}_{x,z}\circ \iota (f)\delta_{y}=\sum_{(u,g,y)\in
      G_{T}\restr{Y(x)}} z^{g}f(u,g,y)\delta_{u}.
  \end{equation*}
  Lemma~\ref{lem-Ht-closed} shows that for $n\in H(x)$,
  \begin{equation*}
    \sum_{(u,g,y)\in
      G_{T}\restr{Y(x)}} z^{g}f(u,g,y)\delta_{u}= \sum_{(u,g+n,y)\in
      G_{T}\restr{Y(x)}} z^{g}f(u,g,y)\delta_{u}.
  \end{equation*}
  As in Lemma~\ref{lem-cts-linear}, for $\phi\in\cc(H(x))$ and $f\in
  \cc(G_{T}\restr {Y(x)})$, we have $\pi_{x,z}\circ\iota (\phi\cdot
  f)=\hat\phi(z)(\pi_{x,z}\circ \iota)(f)$ and $\pi_{y,w}\circ \iota
  (\phi\cdot f)=\hat\phi(w) (\pi_{y,w}\circ \iota )(f)$.  Choose
  $\phi$ such that $\hat\phi(w)=0$ and $\hat\phi(z)\not=0$, and choose
  $f \in \cc(Y(x))$ such that $f(x) = 1$. Then
  $\pi_{y,w}\circ \iota(\phi\cdot f) = 0$ whereas
  $\pi_{x,z}(\phi \cdot f)\delta_x = \hat\phi(z)\delta_x \not= 0$.
  So the kernels are not equal.

  Third, suppose that $\overline{[x]} = \overline{[y]}$ and
  $\overline{z}w \in H(x)^\perp$.  Again Lemma~\ref{lem-y-separate}
  shows that $Y(x) = Y(y)$.  Let ${\pi}_{[x]}$ and ${\pi}_{[y]}$ be
  the faithful irreducible representations of
  $C^*(G_{T|_{Y(x)}}/\Ii(x)) = C^*(G_{T|_{Y(y)}}/\Ii(y))$ described by
  Lemma~\ref{lem:augmentations irreducible}. It is routine to check
  that $\pi^0_{x, z} \circ \iota = \omega_{[x]} \circ \varepsilon_z
  \circ \Phi$ and $\pi^0_{y, w} \circ \iota = \omega_{[y]} \circ
  \varepsilon_w \circ \Phi$. We have
  \begin{equation*}
    \omega_{[x]} \circ \varepsilon_z \circ \Phi \circ \tilde\alpha_{\overline{z}w}
    = \omega_{[x]} \circ \varepsilon_w \circ \Phi.
  \end{equation*}
  Since $\tilde\alpha_{\overline{z}w}$ is an automorphism, we deduce
  that $\ker(\omega_{[x]} \circ \varepsilon_z \circ \Phi) =
  \ker(\omega_{[y]} \circ \varepsilon_w \circ \Phi)$. Thus
  $\ker(\pi^0_{x, z} \circ \iota) = \ker(\pi^0_{y, w} \circ \iota)$.
  Now Corollary~\ref{cor-res-to-y}(\ref{it:circiota}) implies that
  $\pi^0_{x,z}$ and $\pi^0_{y,w}$ have the same kernel. Since
  $\overline{[x]} = \overline{[y]}$, we have $R_x = R_y$, and so
  \begin{equation*}
    \ker \pi_{x,z}
    = R_{x}^{-1}(\ker\pi^0_{x,z})
    = R_{y}^{-1}(\ker\pi^0_{y,w})
    = \ker \pi_{y,w}.
  \end{equation*}

  It remains to show that $(x,z) \mapsto \ker\pi_{x,z}$ is
  surjective. Fix a primitive ideal $I \lhd C^*(G_T)$.
  Theorem~\ref{thm:all irreps} gives $I = \ker
  \pi_{x, z, \rho}$ for some $x$,$z$, $\rho$. Choose $y \in [x]
  \cap Y(x)$, and let $\tilde\omega_{[y]}$ be the faithful irreducible
  representation of $C^*(G_{T|_{Y(x)}}/\Ii(x))$ of
  Lemma~\ref{lem:augmentations irreducible}. Since $\rho$ is
  faithful on $C^*(G_{T|_{Y(x)}}/\Ii(x)))$, we have
  $\ker(\omega_{[y]} \circ \varepsilon_z \circ \Phi) = \ker(\rho \circ
  \varepsilon_z \circ \Phi)$.  So Theorem~\ref{thm:all irreps}
  gives $\ker \pi_{x, z, \omega_{[y]}} = \ker(\pi_{x, z,
  \rho})$. As in the second step above, one checks on basis
  elements that $\pi_{x,z} = \pi_{x, z, \omega_{[y]}}$, completing the proof.
\end{proof}

%%%%%%%%%%%%%%%%%%%%%%%%%%%%%%%%%%
%%%%%%% End Matter %%%%%%%%%%%%%%%
%%%%%%%%%%%%%%%%%%%%%%%%
% \bibliographystyle{amsxport}
%  %
% \bibliography{references-nov01}

\def\noopsort#1{}\def\cprime{$'$} \def\sp{^}

\end{document}